\documentclass[12pt]{amsart}

\usepackage{hyperref}
\usepackage{amsmath,amssymb,latexsym,amsthm}
\usepackage{url}
\usepackage{graphicx,color}
\newtheorem{lemma}{Lemma}
\newtheorem{proposition}{Proposition}
\newtheorem{theorem}{Theorem}

\newtheorem{example}{Example}
\newtheorem{corollary}{Corollary}

\def\C{\mathbb C}
\def\R{\mathbb R}

\def\N{\mathbb N}
\def\p{\partial}
\def\O{\mathcal O}
\def\sup{\mathop{\rm sup}}
\def\supp{\mathop{\rm supp}}

\def\ll{\left}
\def\rr{\right}

\newcommand{\norm}[1]{\left\|#1 \right\|}

\usepackage[margin=0.5cm]{caption}
\usepackage{subcaption}
\usepackage{transparent}

\usepackage{color}

\def\grad{\text{grad}}
\def\bar{\overline}

\def\tilde{\widetilde}
\def\inter{\text{int}}

\def\Src{\mathcal S}

\def\Rec{\mathcal R}
\def\Rem{\mathcal V}
\def\diam{\operatorname{diam}}
\def\div{\operatorname{div}}
\def\grad{\operatorname{grad}}

\renewcommand{\div}{\hbox{div}}

\newcommand{\hiddenfootnote}[1]{}%{\footnote{ #1}}

\newcommand{\B}{{\cal B}}  
  
\newcommand{\K}{{\cal K}}

\renewcommand{\O}{{\mathcal O}}

\def\tilde{\widetilde}
\def \bfo {\begin {eqnarray*} }
\def \efo {\end {eqnarray*} }
\def \ba {\begin {eqnarray*} }
\def \ea {\end {eqnarray*} }
\def \beq {\begin {eqnarray}}
\def \eeq {\end {eqnarray}}
\def \supp {\hbox{supp}\,}

\def \diam {\hbox{diam }}

\def \dist {\hbox{dist}}

\def\bra{\langle}
\def\cet{\rangle}

\def \p {\partial}

\title{Hyperbolic inverse problem with data on disjoint sets}

\author[Y. Kian]{Yavar Kian}
\address{Aix Marseille Univ, Université de Toulon, CNRS, CPT, Marseille, France}
\email{yavar.kian@univ-amu.fr}

\author[Y. Kurylev]{Yaroslav Kurylev}
\address{Department of Mathematics, University College London, Gower Street, London UK, WC1E 6BT}
\email{y.kurylev@ucl.ac.uk}

\author[M. Lassas]{Matti Lassas}
\address{University of Helsinki, Department of Mathematics and Statistics, P.O. Box 68, FI-00014, Finland}
\email{Matti.Lassas@helsinki.fi}

\author[L. Oksanen]{Lauri Oksanen}
\address{Department of Mathematics, University College London, Gower Street, London UK, WC1E 6BT}
\email{l.oksanen@ucl.ac.uk}

\date{\today}
\begin{document}
\begin{abstract}
We consider a restricted Dirichlet-to-Neumann map $\Lambda_{\Src, \Rec}^T$ associated with the operator
$\p_t^2 - \Delta_g + A + q$
where $\Delta_g$ is the Laplace-Beltrami 
operator 
of a Riemannian manifold $(M,g)$, 
and $A$ and $q$ are a vector field and a function on $M$.
The restriction $\Lambda_{\Src, \Rec}^T$
corresponds to the case where the Dirichlet traces are supported on $(0, T) \times \Src$
and the Neumann traces are restricted on $(0, T) \times \Rec$.
Here $\Src$ and $\Rec$ are open sets, which may be disjoint, on the boundary of $M$.
We show that $\Lambda_{\Src, \Rec}^{T}$ determines uniquely, up the natural gauge invariance, the lower order terms $A$ and $q$  
in a neighborhood of the set $\Rec$
assuming that 
$\Rec$ is strictly convex and that
the wave equation is exactly controllable from $\Src$ in time $T/2$. 
We give also a global result under a convex foliation condition.
The main novelty is the recovery of $A$ and $q$
when the sets $\Rec$ and $\Src$ are disjoint.
We allow $A$ and $q$ to be non-self-adjoint,
and in particular, the corresponding physical system may have dissipation of energy.
\end{abstract}
\maketitle

\section{Introduction}

Let $(M, g)$ be a smooth, connected and compact Riemannian manifold of dimension $n$
with nonempty boundary $\p M$,
let $A$ 
be a smooth complex valued vector field on $M$, 
and let $q$ be a smooth complex valued function on $M$.
We consider the wave equation with Dirichlet data $f \in C_0^\infty((0, \infty) \times \p M)$,
\beq
\label{eq_wave}
\begin{cases}
(\p_t^2 - \Delta_g + A(x) + q(x))u(t,x) = 0, &\text{in $(0, \infty) \times M$},
\\ 
u|_{(0, \infty) \times \p M} = f,  &\text{in $(0, \infty) \times \p M$},
\\ 
u|_{t = 0} = \p_t u|_{t = 0} = 0, &\text{in $M$},
\end{cases}
\eeq
and denote by $u_f= u(t,x)$ the solution of (\ref{eq_wave}).
For open and nonempty sets $\Src, \Rec \subset \p M$ and $T \in (0, \infty]$
we define the response operator,
\ba
\Lambda_{\Src, \Rec}^T : f \mapsto (\p_\nu u_f - \frac{1}{2} (A, \nu)_g u_f)|_{(0,T) \times \Rec}, 
\quad f \in C_0^\infty((0, T) \times \Src).
\ea
Here $\nu$ is the exterior unit normal vector field on $\p M$, and $(A, \nu)_g$ is the inner product of $A$ and $\nu$.
We use real inner products throughout the paper. If $A(x) =\sum_{j=1}^n A^j(x)\p_j$ in local coordinates, then 
$(A, \nu)_g$ is given by $g_{jk}A^j\nu^k$ locally.

When $f$ is regarded as a boundary source, the operator $\Lambda_{\Src, \Rec}^T$ models 
boundary measurements for the wave equation with sources on the set $(0, T) \times \Src$
and the waves being observed on $(0, T) \times \Rec$.  
We consider the inverse boundary value problem to determine the
manifold $(M, g)$, the vector field $A$ and the potential $q$ from $\Lambda_{\Src, \Rec}^T$. 

We have studied previously the determination of the geometry 
$(M, g)$
in the case that $A = 0$ and $q = 0$, see \cite{Lassas2012}.
This corresponds to the recovery of the leading order terms in the wave equation up to isometries.
It appears to us that when the sets $\Src$ and  $\Rec$ are disjoint, the recovery of the lower order terms
$A$ and $q$ is more complicated than the recovery of the leading order terms, and this is the main focus of the present paper. 
It should be emphasized that if $\Src$ and  $\Rec$ are disjoint and
if $\Lambda_{\Src,\Rec}^\infty$ is known only at a fixed frequency in the sense that its Fourier transform in time is given at fixed point, then it is not possible to recover even the geometry $(M,g)$, see \cite{Daude2015}.    

In order to recover $A$ and $q$, we develop a new technique that is based on exploiting strict convexity of the set $\Rec$. We recall that $\Rec$ is said to be strictly convex if its second fundamental form is positive definite, that is, $II(v,v) > 0$ for all $v \in T_x \Rec$ and $x \in \Rec$, where
$$
II(v,v) = (\nabla_v \nu, v)_g,     
$$
and $\nabla$ is the covariant derivative on $(M,g)$.
Using the convexity we construct a boundary source $f$ such that, at time $t=T$, the corresponding solution $u_f$ is essentially localized in a small set near $\mathcal R$.

The lower order terms $A$ and $q$ can be determined only up to the action of a group gauge transformations, that we will describe next.
Let $\kappa$ be a smooth nowhere vanishing complex valued function on $M$ satisfying $\kappa = 1$ on $\Rec$.
The response operator $\Lambda_{\Src, \Rec}^T$ does not change under the  transformation $(A, q) \mapsto (A_\kappa, q_\kappa)$ where
\begin{equation}
\label{gauge_A_Q}
A_\kappa = A + 2 \kappa^{-1} \grad_g \kappa, \quad 
q_\kappa = q + \kappa (A - \Delta_g) \kappa^{-1},	
\end{equation}
and $\grad_g$ is the gradient on $(M,g)$.
We refer to \cite{Kurylev1995} for a similar computation in the self-adjoint case. %This is the gauge invariance.
When $\mathcal U\subset M$ is a neighborhood of $\Rec$, we write
\ba
{\mathcal G}_{\mathcal U,\Rec}(A,q)=\{(A_\kappa|_{\mathcal U},q_\kappa|_{\mathcal U});\ \kappa\in C^\infty(\overline{\mathcal U}),\ \kappa \ne 0,\ \kappa|_\Rec=1\}
\ea
for the orbit of the group of gauge transformations on $\mathcal U$.

%This is the diffeomorphism invariance.

We recall that the wave equation (\ref{eq_wave}) is said to be exactly controllable from $\Src$ in time $T$ if 
%there is closed $\Src \subset \Src_0$ and $T < T_0$ such that 
the map
\begin{equation}
\label{def_exact_controllability}
f \mapsto (u_f(T), \p_t u_f(T)) : L^2((0,T) \times \Src) \to L^2(M) \times H^{-1}(M),
\end{equation}
is surjective. 
If there is such $T > 0$, then we say that (\ref{eq_wave}) is exactly controllable from $\Src$.
The exact controllability can be characterized in terms of the billiard flow of the manifold $(M,g)$ \cite{Bardos1992, Burq1997}.
The geometric characterization says roughly that all unit speed geodesics, continued by reflection on $\p M \setminus \Src$, must exit $M$ through $\Src$ during time $T$.
In particular, the geometric characterization implies that the exact controllability does not depend on the lower order terms $A$ and $q$.

In this paper we show the following theorem:
% In \cite{Rakesh2011} Rakesh and Sacks consider a problem with a similar measurement geometry.
% They consider the formally determined problem where the response $\Lambda^T f$
% is known only for a single point source $f$. % = \delta_x$, $x \in M$.
% The problem we consider here is overdetermined, however, 
% we do not need the angular control assumption of \cite{Rakesh2011}. 
% Our main theorem is the following:

\def\K{\mathcal K}
\begin{theorem}
\label{th_main}
Let $\Src \subset \p M$ be open and suppose that the wave equation (\ref{eq_wave}) is exactly controllable from $\Src$ in time $T > 0$.
Let $\Rec \subset \p M$ be open and strictly convex. For $j=1,2$, we fix $A_j\in C^\infty(M;TM)$  $q_j\in C^\infty(M;\mathbb C)$, we denote by $\Lambda_{j,\Src, \Rec}^{2T}$ the response operator at time $2T$ associated with \eqref{eq_wave} with $A=A_j$ and $q=q_j$. Then assuming that  
\begin{equation}\label{t1a} \Lambda_{1,\Src, \Rec}^{2T}=\Lambda_{2,\Src, \Rec}^{2T},\end{equation}
there is a neighborhood $\mathcal U \subset M$ of $\Rec$, independent of $(A_1,q_1)$ and $(A_2,q_2)$,  such that 
\begin{equation}\label{t1b} (A_1|_{\mathcal U},q_1|_{\mathcal U})\in {\mathcal G}_{\mathcal U,\Rec}(A_2,q_2). \end{equation}
%that is, these data determine the orbit 
%
% $(A, Q) \mapsto (A_\kappa, Q_\kappa)$ where
%the complex valued function $\kappa$ is in
%$$ 
%\{ \kappa \in C^\infty(\mathcal U);\ \kappa|_{\Rec} = 1,\ \kappa(x) \ne 0 \text{ for all $x \in \mathcal U$} \}.
%$$
\end{theorem}

We show also a global uniqueness result
under the assumption that there is a convex foliation 
similar to that in 
\cite{Stefanov2013a}.
We assume that $\Sigma_s$, $s \in (0,1]$, satisfy the following:
\begin{itemize}
\item[(F1)] $\Sigma_s \subset M^\inter$ is a smooth manifold of codimension one.
\item[(F2)] 
The union $\Omega_s = \bigcup_{r \in (0,s)} \Sigma_r \subset M^\inter$ is open and connected,
and $\Omega_r \subset \Omega_s$ when $r < s$.
\item[(F3)] $\p \Omega_s = \Sigma_s \cup \Rec_s$ 
and $\Rec_s \subset \Rec$ where 
$\Rec_s = \overline{\Omega_s} \cap \p M.$
\item[(F4)] $\Sigma_s$ is strictly convex as a subset of $\p M_s$ where  
$$M_s = M \setminus (\Omega_s \cup \Rec_s).$$
\item[(F5)] The Hausdorff distances satisfy $\dist(\Omega_r, \Omega_s) \to 0$ as $r \to s$.
\item[(F6)] There is a set $\Rec_0 \subset \Rec$ 
such that $\dist(\Omega_s, \Rec_0) \to 0$ as $s \to 0$.
\end{itemize}

Furthermore, to simplify the notation, we assume
\begin{itemize}
\item[(F7)] $\Rec = \bigcup_{s \in (0,1]} \Rec_s$.
\end{itemize}

\begin{theorem}
\label{th_main_foliation}
Let $\Src \subset \p M$ be open and suppose that the wave equation (\ref{eq_wave}) is exactly controllable from $\Src$.
Let $\Rec \subset \p M$ be open and strictly convex and
let $\Sigma_s$, $s \in (0, 1]$, be a convex foliation satisfying (F1)-(F7).
Then \eqref{t1a} implies  that there exists  $\mathcal U \subset M$ an open set of $M$containing $\bar \Omega_1$ such that
\begin{equation}\label{t2a} (A_1|_{\mathcal U},q_1|_{\mathcal U})\in {\mathcal G}_{\mathcal U,\Rec}(A_2,q_2). \end{equation}
 
\end{theorem}

In Section \ref{sec_complementary_results} we show that, in the above theorem,
exact controllability from $\Src$ can be replaced with exact controllability from $\Rec$.
Our result is new even in the following case:

\begin{example}
Let $(M,g)$ be the Euclidean unit disk $\{z \in \C;\ |z| \le 1\}.$ Let 
$\epsilon > 0$ and define $\Rec = \{e^{i\theta}; \theta \in (-\epsilon, \pi + \epsilon)\}$. Let $\Src \subset \p M$ be open and nonempty.
Then $\Lambda_{\Src, \Rec}^\infty$ determines 
$A$ and $q$, up to the gauge transformations, in the convex hull of $\Rec$.
\end{example}
\begin{example}
Let $(M,g)$ be the Euclidean annulus  $$\{z \in \C;\ r_1\le|z| \le r_2\}, \quad  \text{with $0<r_1<r_2$}.$$ Let 
$\epsilon > 0$ and define $\Rec = \{r_2e^{i\theta}; \theta \in (-\epsilon, \pi + \epsilon)\}$. Let $\Src \subset \p M$ be open and nonempty and assume that $\Src$ satisfies the geometrical control condition of
\cite{Bardos1992},
for example, $\Src = \{z \in \C;\ |z|=r_1\}$.
Then $\Lambda_{\Src, \Rec}^\infty$ determines 
$A$ and $q$, up to the gauge transformations, in the convex hull of $\Rec$. %Note that here we may both allow $\Src\subset \{z \in \C;\ |z| = r_2\}$ or $\Src\subset \{z \in \C;\ |z| = r_1\}$.
\end{example}

Let us also point out that we could use a time continuation argument analogous to \cite[Lemma 4]{Lassas2012}
and prove Theorem \ref{th_main_foliation} also for measurements on a long enough but finite time interval. 

Our proof is based on the Boundary Control (BC) method.
The BC method was introduced by Belishev \cite{Belishev1987}, and it was first used in a geometric context in \cite{Belishev1992}.
Stability properties of the method are discussed in \cite{Anderson2004} and in the recent preprint \cite{Bosi2017}.
First order perturbations have been considered in 
the self-adjoint case in \cite{Katchalov1998, Kurylev2015}, and in the non-self-adjoint case in
\cite{Kurylev2000a, Kurylev1997c}. All the above results assume that $\Src = \Rec$.
The case of disjoint $\Src$ and $\Rec$ was first considered in the above mentioned \cite{Lassas2012} where no first order perturbation was present. 

In addition to \cite{Lassas2012}, we are aware of only two results on inverse boundary value problems with disjoint data analogous to the case $\overline\Src \cap \overline\Rec = \emptyset$.
Rakesh \cite{Rakesh2000} considers a wave equation 
on a one-dimensional interval with 
sources supported on one end of the interval and the waves observed on the other end,
and Imanuvilov, Uhlmann, and Yamamoto \cite{Imanuvilov2011a} proved that 
a zeroth order term in a Schr\"odinger equation 
on a two-dimensional domain homeomorphic to a disk, whose boundary
is partitioned into eight parts $\Gamma_1,\Gamma_2,\dots,\Gamma_8$ in the clockwise order, is determined by boundary measurements with Dirichlet data supported on $\Src=\Gamma_2\cup\Gamma_6$ and the Neumann trace observed on $\Rec=\Gamma_4\cup\Gamma_8$.

Let us mention also the result on recovery of a conformal scaling factor in the metric tensor given the Dirichlet-to-Neumann map \cite{Stefanov2015}
that, analogously to our result, uses local convexity of the boundary.
The proof \cite{Stefanov2015} is based on a reduction to the boundary rigidity result \cite{Stefanov2013c} and this approach seems to require that $\Src = \Rec$.
Let us also mention that, using the equivalence between  inverse problems for hyperbolic equation and inverse problems stated for other equations described in \cite{Katchalov2004,KOSY},  the result of this paper can also be applied in other context for different types of equations. 

A vast majority of results on inverse boundary value problems 
assume that $\overline\Src \cap \overline\Rec \ne \emptyset$.
For this type of non-disjoint, partial data results, we refer to 
\cite{Caro2009, Chung2014, DosSantosFerreira2007, Greenleaf2001a, Guillarmou2011, Imanuvilov2012, Imanuvilov2010b, Kenig2013}.

\subsection{Outline of the paper}

In Section 2.1 we first prove a version of Blagove{\v{s}}{\v{c}}enski{\u\i}'s identity that, given the response operator $\Lambda_{\Src,\Rec}^{2T}$, allows us to compute inner products of the form
    \begin{equation}\label{outline1}
\bra v_\phi(T,\cdot) , u_f(T,\cdot) \cet _{L^2(M)},
    \end{equation}
where $u_f$ is the solution of (\ref{eq_wave}), with $f$ supported in $(0,2T) \times \Src$,
and $v_\phi$ is the solution of a dual problem, see (\ref{eq_wave_adjoint1}) below, with the source $\phi$ supported in $(0,2T) \times \Rec$.

It follows from approximate controllability, as recalled in Section 2.2, that if for a fixed $f$ the inner products (\ref{outline1}) vanish for all $\phi$ in $C_0^\infty((0,2T) \times \Rec)$, then $u_f(T,\cdot)$ must be supported outside the set of points $x \in M$
satisfying $d(x,\Rec) < T$. 

A key idea in the proof is to use a more refined argument to show that for a point $x \in M$ near $\Rec$ we can choose a set $\mathcal B \subset (0,2T) \times \Rec$ such that if 
(\ref{outline1}) vanish for all $\phi$ in $C_0^\infty(\mathcal B)$, then $u_f(T,\cdot)$ must be supported outside a set that isolates a neighbourhood of $x$ as in Figure \ref{fig_Mh} in Section \ref{sec_convexity}.
This construction uses the convexity of $\Rec$ and it is carried out in Section 3.1.

Using the assumption that the wave equation (\ref{eq_wave}) is exactly controllable from $\Src$ in time $T > 0$, we then show in Section 3.2 that it is possible to construct a sequence of sources $(f_k)_{k \in \N}$ in 
$L^2((0,T) \times \Src)$ 
so that, when restricted near $\Rec$, the functions $u_{f_k}(T,\cdot)$ converge to a point mass $\kappa \delta_x$ at $x$, with $\kappa$ a nonzero constant.
 
We finish the proof of Theorem \ref{th_main} in Section 3.3. The idea is to substitute $u_{f_k}(T,\cdot)$ in (\ref{outline1}) and recover $\kappa v_\phi(T,x)$ after passing to the limit. Here the constant $\kappa$ depends on the choice of the point $x$, and some technical work is required in order to recover $\kappa(x) v_\phi(T,x)$, with $\kappa$ a smooth function near $\Rec$. 
Having the functions $\kappa(x) v_\phi(T,x)$ at hand, it is straightforward to find out which wave equation they satisfy. This will give us $A_\kappa$ and $q_\kappa$
as in (\ref{gauge_A_Q}).

The global result in Theorem \ref{th_main_foliation} is proven in Section 4 by iterating the local construction along the convex foliation. Finally in Section 5, we give some complementary results.  

\section{Tools for the inverse problem}

In this section we present the two main components of the Boundary Control method:
an integration by parts technique 
originating from Blagove{\v{s}}{\v{c}}enski{\u\i}'s study of the $1+1$ dimensional wave equation \cite{Blagovescenskiui1971},
and a density result based on the hyperbolic unique continuation result by Tataru \cite{Tataru1995}.

\subsection{Blagove{\v{s}}{\v{c}}enski{\u\i}'s identity}
\label{sec_Blago}

Let $\Gamma \subset \p M$ and $B \subset M^\inter$ be open, and 
let $\kappa : B \to \C$ be a smooth function satisfying 
$$\kappa(x)\neq 0,\quad x\in B.$$
We define for $f \in \mathcal C_0^\infty((0, \infty) \times S)$, 
$$
\Lambda_\Gamma f = (\p_\nu u - \frac{1}{2} (A, \nu)_g u)|_{(0,\infty) \times \Gamma},
\quad \mathcal T_{B, \kappa} f =  \kappa u |_{(0,\infty) \times B},
$$
where $u$ is the solution of (\ref{eq_wave}), and
write $\mathcal T_B = \mathcal T_{B, \kappa}$ when considering a fixed $\kappa$.

For all $\tau>0$ we define  $K_\Gamma^\tau$ and $K_{B,\kappa}^\tau$, for any $f\in C_0^\infty((0, \infty) \times \Src)$, $(t,x)\in (0,+\infty)\times\partial M$ and $(t,y)\in (0,+\infty)\times B$ , by
\begin{eqnarray}
\label{def_K_Gamma}
K_\Gamma^\tau f(t,x) &=& \Lambda_\Gamma  J^\tau f(\tau-t,x)-\tilde J^\tau\Lambda_\Gamma f(t,x) ,
\\ \label{def_K_B}
K_{B,\kappa}^\tau f (t,y) &=& \mathcal T_{B, \kappa}  J^\tau f(\tau-t,y)-\tilde J^\tau\mathcal T_{B, \kappa} f(t,y) ,
\end{eqnarray}
where
$J^\tau\psi(t,x) = \frac{1}{2} \int_{\tau-t}^{t+\tau} \psi(s,x) ds$ and $\tilde J^\tau\psi(t,x) = \frac{1}{2} \int_t^{2\tau - t} \psi(s,x) ds$.

Let us now %suppose that $\kappa \in C^\infty(\overline B)$ and 
consider the adjoint problems
\beq
\label{eq_wave_adjoint1}
\begin{cases}
\p_t^2v - \Delta_gv -\overline{A}v + (\overline{q}- \div_g \overline{A})v= 0, &\text{in $(0, \infty) \times M$},
\\ 
v|_{(0, \infty) \times \p M} = \phi,  &\text{in $(0, \infty) \times \p M$},
\\ 
v|_{t = 0} = \p_t v|_{t = 0} = 0, &\text{in $M$},
\end{cases}
\eeq

\beq
\label{eq_wave_adjoint2}
\begin{cases}
\p_t^2w - \Delta_gw -\overline{A}w + (\overline{q}- \div_g \overline{A})w=   H, &\text{in $(0, \infty) \times M$},
\\ 
w|_{(0, \infty) \times \p M} = 0,  &\text{in $(0, \infty) \times \p M$},
\\ 
w|_{t = 0} = \p_t w|_{t = 0} = 0, &\text{in $M$},
\end{cases}
\eeq
where $H \in C_0^\infty((0,+\infty)\times B)$,  $\phi \in C_0^\infty((0,\infty) \times \p M)$, $\div_g$  the divergence on $(M,g)$. We fix $v_f$ (resp. $u_f$, $w_H$) the unique solution of \eqref{eq_wave_adjoint1} (resp. \eqref{eq_wave}, \eqref{eq_wave_adjoint2}) lying in $C^1((0,+\infty);L^2(M))\cap C((0,+\infty);H^1(M))$.
Now let us consider the following identity

\begin{lemma}[Blagove{\v{s}}{\v{c}}enski{\u\i} type identity]
\label{lem_blago}
Let $\tau > 0$ and 
let $\Gamma \subset \p M$ and $B \subset M^\inter$ be open. 
Then for functions 
%\begin{eqnarray*}
$f \in L^2_{loc}((0, +\infty) \times \Src)$, $\phi \in L^2_{loc}((0, \infty) \times \Gamma)$, % \quad
$H \in L^2_{loc}((0, +\infty) \times B)$
%\end{eqnarray*}
we have
\begin{eqnarray}
\label{blago_inner_prod1}
\bra v_\phi(\tau,\cdot) , u_f(\tau,\cdot) \cet _{L^2(M)} 
= \bra \phi , K_\Gamma^\tau f\cet _{L^2((0, \tau) \times \Gamma)}.
\end{eqnarray}
\begin{eqnarray}
\label{blago_inner_prod2}
\bra w_{\overline{\kappa}H}(\tau,\cdot) , u_f(\tau,\cdot) \cet _{L^2(M)} 
= \bra H , K_{B,\kappa}^\tau f\cet _{L^2((0, \tau) \times \Gamma)}.
\end{eqnarray}
\end{lemma}
\begin{proof} Since the proof of \eqref{blago_inner_prod1} and \eqref{blago_inner_prod2} are similar, we will only treat \eqref{blago_inner_prod1}. Without loss of generality and by density, we assume that $f \in C^\infty_0((0, +\infty) \times \Src)$, $\phi \in C^\infty_0((0, \infty) \times \Gamma)$.
For $t \in (0,\tau)$ and $s \in (0, 2\tau)$,  we start by considering $$S(t,s)=\bra v_\phi(t,\cdot) , u_f(s,\cdot) \cet _{L^2(M)}.$$
Recall that 
\ba
&&(\p_t^2 - \p_s^2) S(t,s)
\\&&\quad=  
\bra  (\Delta_g +  \overline{A} - (\overline{q}- \div_g \overline{A}))  v_\phi(t,\cdot), u_f(s,\cdot)\cet _{L^2(M)}
\\&&\qquad-\bra v_\phi(t,\cdot),(\Delta_g - A - q) u_f(s,\cdot)\cet _{L^2(M)}
\\&&\quad= \bra  \p_\nu v_\phi(t,\cdot) + \frac 1 2 (\overline{A}, \nu)_g v_\phi(t,\cdot),f(s,\cdot)\cet _{L^2(\p M)}
 \\&&\qquad-\bra \phi(t,\cdot),\Lambda_{\Gamma} f(s,\cdot)\cet _{L^2(\p M)}.
\ea

Thus, fixing
$$\begin{aligned}
F(t,s)=& \bra  \p_\nu v_\phi(t,\cdot) + \frac 1 2 (\overline{A}, \nu)_g v_\phi(t,\cdot),f(s,\cdot)\cet _{L^2(\p M)}\\
 \ &-\bra \phi(t,\cdot),\Lambda_{\Gamma} f(s,\cdot)\cet _{L^2(\p M)},
\end{aligned}$$
  we deduce that the function $S$
satisfies the $1+1$ dimensional wave equation 
$$\left\{\begin{array}{ll}\partial_t^2S-\p_s^2S=F,\quad &\textrm{in}\ (0,\tau)\times(0,2\tau),\\  S(0,\cdot)=0,\quad \partial_tS(0,\cdot)=0,\quad &\textrm{in}\ (0,\tau),\\ S(\cdot,0)=0,\quad \partial_sS(\cdot,0)=0,\quad &\textrm{in}\ (0,2\tau).\end{array}\right.$$

We solve this wave equation on the triangle with corners $(\tau,\tau)$, $(0,0)$ and $(0,2\tau)$,  and obtain
\begin{equation}\label{blago_step_T}
\begin{aligned}
&\bra v_\phi(\tau,\cdot) , u_f(\tau,\cdot) \cet _{L^2(M)}\\
&= \frac{1}{2}
 \int_0^\tau \int_{t}^{2\tau-t} 
\bra  \p_\nu v_\phi(t,\cdot) + \frac 1 2 (\overline{A}, \nu)_g v_\phi(t,\cdot),f(s,\cdot)\cet _{L^2(\p M)} ds dt\\
&-\frac{1}{2} \int_0^\tau \int_{t}^{2\tau-t} 
\bra \phi(t,\cdot),\Lambda_{\Gamma} f(s,\cdot)\cet _{L^2(\p M)} ds dt.
\end{aligned}\end{equation}
 Now, for $h \in C^\infty_0((-\infty,\tau) \times \Src)$,  let  $u=R_\tau u_{R_\tau h}$  and $v= v_{\phi}$ where $R_\tau$ is the time reversal operator defined by
$R_\tau h(t,x)=h(\tau-t,x)$.
Then
$$\begin{aligned} &\bra h, \p_\nu v + \frac 1 2 (\overline{A}, \nu)_g v\cet _{L^2((0,T) \times \p M)}- \bra \p_\nu u - \frac 1 2 (A, \nu)_g u, \phi\cet_{L^2((0,\tau) \times \p M)}
\\ &=\bra (\p_t^2 - \Delta_g + A + q)u, v\cet_{L^2((0,\tau) \times M)}
\\ &\qquad
- \bra u, (\p_t^2 - \Delta_g - \overline{A} + (\overline{q}-\textrm{div}_g \overline{A})) v\cet _{L^2((0,\tau) \times M)}
=0.\end{aligned}$$

%{\footnotesize
%Indeed (recall that $\nu$ is the {\bf interior} unit normal)
%\ba
%((A + q)u^f, w) &= -(u, \bar A w) - (u, \div \bar A, w) - ((A, \nu)u, w)_{x \in \p M} + (qu,w)
%\\&= (u, (- \bar A + \bar q - \div \bar A)w) + (-\frac 1 2 (A, \nu)u, w)_{x \in \p M} - (u, \frac 1 2 (\bar A, \nu) w)_{x \in \p M}.
%\ea
%}

% with a Dirichlet boundary source $\phi \in C_0^\infty((0,T) \times \Rem)$ and vanishing initial conditions.
% Then the integration by parts (\ref{green}) gives 
% $$
% (\Lambda^T_{\Src,\Rem}f, \phi)_{L^2((0,T) \times \Rem)} 
% = (f, \p_\nu w + \frac 1 2 (\bar A, \nu) w)_{L^2((0,T) \times \Src)}, 
% $$
% for all $f \in C_0^\infty((0,T) \times \Src)$.
Therefore, we have
$$\bra h, \p_\nu v_\phi + \frac 1 2 (\overline{A}, \nu)_g v_\phi\cet _{L^2((0,\tau) \times \p M)}=\bra R_\tau\Lambda_\Gamma R_\tau h, \phi\cet_{L^2((0,\tau) \times \p M)}.$$
Fixing
$$h(t,\cdot):=\int_{t}^{2\tau-t} f(s,\cdot)ds$$
we find
$$\begin{aligned} &\frac{1}{2}
 \int_0^\tau \int_{t}^{2\tau-t} 
\bra  \p_\nu v_\phi(t,\cdot) + \frac 1 2 (\overline{A}, \nu)_g v_\phi(t,\cdot),f(s,\cdot)\cet _{L^2(\p M)} ds dt\\
&=\frac{1}{2}
 \int_0^\tau 
\bra  \p_\nu v_\phi(t,\cdot) + \frac 1 2 (\overline{A}, \nu)_g v_\phi(t,\cdot),h(t,\cdot)\cet _{L^2(\p M)}  dt\\
&=\frac{1}{2}
 \int_0^\tau 
\bra \phi(t,\cdot),R_\tau\Lambda_\Gamma R_\tau h(t,\cdot)\cet _{L^2(\p M)}  dt\\
&=\int_0^\tau 
\bra \phi(t,\cdot),R_\tau\Lambda_\Gamma J^\tau f(t,\cdot)\cet _{L^2(\p M)}  dt.\end{aligned}$$
Combining this with \eqref{blago_step_T}, we deduce \eqref{blago_inner_prod1}.

\end{proof}

\subsection{Approximate controllability}

Next we consider approximate controllability on a domain of influence. 
Let $T>0$, $\Gamma \subset \p M$ and $B \subset M^\inter$ be open,
and let $\Rem = \Gamma$ or $\Rem = B$.
Let $h : \bar \Rem \to \R$ be piecewise continuous, and
define the domain of influence
\ba
M(\Rem, h) = \{x \in M;\ \inf_{y \in \Rem} (d(x, y) - h(y)) \le 0\},
\ea
where $d$ is the distance function on $(M,g)$.
Moreover, we write 
\def\B{\mathcal B}
\ba
\B(\Rem, h; T) = \{(t, y) \in (0, \infty) \times \Rem;\ T - h(y) < t \}. 
\ea
We extend the notations $M(\Rem, h)$ and $\B(\Rem, h; T)$
for constants $h \in \R$ by interpreting $h$ as a constant function. 
Moreover, we define $M(x, h)$ by $M(\{x\}, h)$ for points $x \in \p M$.

We have the following approximate controllability result that is analogous to \cite[Lemma 5]{Lassas2012} and \cite[Lemma 2.5]{Kurylev2015}.

\begin{lemma}\label{Lem. Tataru}
Let $T>0$,  $\Gamma \subset \p M$ and $h : \Gamma\to \R$ to be piecewise continuous. Then, the set
$$\{ v_\phi(T,\cdot);\ \phi\in\mathcal C^\infty_0(\B(\Gamma, h; T) )\}$$
is dense in 
$
L^2(M(\Gamma, h)) = \{y \in L^2(M);\ \supp(y) \subset M(\Gamma, h) \}.
$
In the same way,   for $B \subset M^\inter$ an open set and $h : \bar B \to \R$ a piecewise continuous function satisfying $h > 0$ pointwise, the set
$$\{ w_H(T,\cdot);\ H\in\mathcal C^\infty_0(\B(B, h; T) )\}$$
is dense in 
$
L^2(M(B, h)) = \{y \in L^2(M);\ \supp(y) \subset M(B, h) \}.
$
\end{lemma}

The $L^2$-topology used in the above lemma does not give control over the point values of $v_\phi$. 
For this reason, we need occasionally also the following lemma, that is analogous to Lemma 3.7 in \cite{Kurylev2015}.

\begin{lemma}\label{lem_non_vanishing}
Let $T>0$,  $\Gamma \subset \p M$, $B\subset M^\inter$ and suppose that functions $h_1 : \Gamma\to \R$ and $h_2 : B \to \R$ are piecewise continuous. 
Let $x$ and $y$ be points in $M(\Gamma, h_1)^\inter$ and $M(B, h_2)^\inter$, respectively.
Then there exist
$\phi\in\mathcal C^\infty_0(\B(\Gamma, h_1; T) )$ and $H\in\mathcal C^\infty_0(\B(B, h_2; T) )$
such that the solution $v$ of \eqref{eq_wave_adjoint1} and the solution $w$ of \eqref{eq_wave_adjoint2} satisfy
$v(T,x) \ne 0$ and $w(T,y) \ne 0$.
\end{lemma}

\section{Local determination of the first order perturbation}

In this section we prove Theorem \ref{th_main}. Before formulating the geometric step of our proof, that is, Proposition \ref{prop_convergence_to_pt} below, let us introduce some notation.
Let $\Gamma \subset \p M$ be open. Then 
the boundary normal coordinates adapted to $\Gamma$ are 
given by the map 
\begin{eqnarray}
\label{boundary_normal_coords_s_y}
(s,y) \mapsto \gamma(s; y, -\nu), \quad y \in \Gamma,\ s \in [0, \sigma_{\Gamma,M}(y)),
\end{eqnarray}
where the cut distance $\sigma_{\Gamma,M} : \Gamma \to (0,\infty)$ is defined by 
\begin{eqnarray} 
\label{def_sigma}
\sigma_{\Gamma,M}(y) &=& \max \{ s \in (0, \tau_M(y)];\ d(\gamma(s; y, -\nu), \Gamma) = s\},
\\\notag
\tau_M(y) &=& \sup \{ s \in (0, \infty);\ \gamma(s; y, -\nu) \in M^\inter \}.
\end{eqnarray}
Here $\gamma(\cdot; x, \xi)$ is the geodesic with the initial data $(x,\xi) \in TM$,
and we recall that $\nu$ is the exterior unit normal on $\p M$.
We often write $\sigma_\Gamma = \sigma_{M,\Gamma}$.
Note that $\sigma_\Gamma(y) > 0$, see e.g. \cite[p. 50]{Katchalov2001}. 

We define $$M_\Gamma = \{\gamma(s;y,  -\nu);\ y \in \Gamma,\ s \in [0, \sigma_\Gamma(y))\}.$$ Then a point $x \in M_\Gamma$ is represented in the coordinates (\ref{boundary_normal_coords_s_y}) by $(s, y)$, 
where $s = d(x,\Gamma)$ and $y$ is the unique closest point to $x$ in $\Gamma$.

We will also use the notations 
$$\begin{aligned}
B(p,r) &= \{x \in M;\ d(x,p) < r\}, \quad p \in M,\ r > 0,
\\
B_{\p M}(y,r) &= \{x \in \p M;\ d(x,y) < r\}, \quad y \in \p M,\ r > 0.
\end{aligned}$$

\subsection{A convexity argument}
\label{sec_convexity}
Our aim is to 
construct a sequence of functions $(h_k)_{k\in\mathbb N}$  on $\Rec$
such that the difference of the domains of influences $M(\Gamma, s) \setminus M(\Rec, h_k)$
converges to a point $x \in M$ as $k \to \infty$.
Here $\Gamma \subset \Rec$ and $s > 0$ will be chosen suitably, and the point $x$ will lie in vicinity of $\Rec$.
We will use this construction to enforce a sequence of solutions of \eqref{eq_wave} to converge, at a fixed time, to a point mass at $x$.
The main result of this subsection can be stated as follows.

\begin{proposition}
\label{prop_convergence_to_pt}
Let $\Gamma \subset \p M$ be open and strictly convex and let $\K \subset \Gamma$ be compact.
Define for $p = (s,y) \in  M_\Gamma$ and small $\epsilon > 0$,
$$
C(p,\epsilon) = \ll((s-\epsilon, s + \epsilon) \cap [0,\infty)\rr) \times B_{\p M}(y, \epsilon),$$
in the coordinates (\ref{boundary_normal_coords_s_y}).
There exist a neighborhood $U  \subset M_\Gamma$ of $\K$ such that for all $p \in U$
there is $\epsilon > 0$ satisfying the following.
%If $\epsilon' > 0$ 
%is small enough so that $B_{\p M}(y, \epsilon')\subset \Gamma$ and $C_p \subset B(p,\delta)$
%in the coordinates \emph{(\ref{boundary_normal_coords_s_y})}, where
For any $x \in C_p = C(p,\epsilon)$ there exists a sequence of functions $(h_{k,x})_{k\in\mathbb N}$ in $C(\overline \Gamma)$ such that the set
$$X_{k,x} =  M(B_{\p M}(y, \epsilon), s+\epsilon)^\inter \setminus M(\Gamma, h_{k,x}),\quad k\in\mathbb N,$$
is a neighborhood of $x$, and $\diam(X_{k,x}) \to 0$ as $k \to \infty$.
\end{proposition}

The functions $h_{k,x}$ are given explicitly by (\ref{def_hkx}) below.
In order to prove this result we will need three intermediate results.

\begin{lemma}
\label{lem_convexity}
Let $\Gamma \subset \p M$ be open and strictly convex, and let $\K \subset \Gamma$ be compact. 
Then there is $\delta(\K) > 0$ and a neighborhood $U(\K) \subset M_\Gamma$ of $\K$ such that, for all $x \in U(\K)$ and $q \in B(x, \delta(\K)) \setminus \{x\}$, there is $z \in \Gamma$ satisfying $d(z,q) < d(z,x)$.
\end{lemma}
\begin{proof}
Let us consider a unit speed geodesic $\gamma(t) = (r(t), z(t))$ in coordinates (\ref{boundary_normal_coords_s_y})
and denote the initial data of $\gamma$ by
$$
\gamma(0) = (s, y),\quad \dot \gamma(0) = (\rho, \eta).
$$ 
We will first  show that there is a neighborhood $U \subset M$ of $\K$ and $\rho_0 > 0$ such that, for all $(s,y) \in \overline U$ and $\rho \in [-1,\rho_0]$,
the geodesic $\gamma$ intersects $\Gamma$ and is distance minimizing until the intersection.

To this end recall that, in coordinates (\ref{boundary_normal_coords_s_y}),
the metric tensor $g$ is of the form 

\begin{equation}
\label{g_in_bncoords}
g(s,y) = \left( \begin{array}{cc}
1 & 0 \\
0 & h(s,y) \\ \end{array} \right).
\end{equation}
We write $(x^1,\dots,x^n) = (s,y)$
and $\p_j = \p_{x^j}$.
Then it follows from (\ref{g_in_bncoords}) that 
the Christoffel symbols $\Gamma_{jk}^l$
satisfy for $\alpha,\beta = 2,\dots,n$,
$$
\Gamma_{\alpha 1}^\beta = \sum_{\kappa=2}^n\frac 1 2 h^{\beta \kappa} \p_1 h_{\kappa \alpha} = 
- \sum_{\kappa=2}^nh^{\beta \kappa} \Gamma_{\alpha \kappa}^1,
$$
and that the (scalar) second fundamental form of $\p M$ 
satisfies
$$
II(\p_\alpha, \p_\beta)(y) = - \sum_{\kappa=2}^nh_{\beta \kappa} \Gamma_{\alpha 1}^\kappa(0,y) = \Gamma_{\alpha \beta}^1(0,y).
$$
The geodesic equations imply that 
\ba
r(t) = s + t \rho - \frac{t^2}{2} \sum_{\alpha,\beta=2}^n\Gamma^1_{\alpha \beta}(s,y) \eta^\alpha \eta^\beta + \O(t^3),
\ea
see e.g. \cite[p. 113]{Sharafutdinov1994}.
Moreover, the strict convexity of $\Gamma$,
the lower semi-continuity of the cut distance function $\sigma_\Gamma$ and the compactness of $\K$ imply that there is a neighborhood $U_0 \subset M_\Gamma$
 of $\K$ and $c>a > 0$ such that, 
for all $(s,y) \in \overline{U_0}$,
\ba
 a |\eta|_{h}^2 \le \sum_{\alpha,\beta=2}^n\Gamma^1_{\alpha \beta}(s,y) \eta^\alpha \eta^\beta \le c |\eta|_{h}^2.
%\quad \text{and} \quad 	y \in \Gamma.
\ea

We will consider only the case $|\eta|_{h}^2> 1/2$.
Note that if $\rho_0 > 0$ is small and $|\eta|_h^2 \le 1/2$, then $\rho < \rho_0$ implies that $\rho < 0$
since $(\rho, \eta)$ is an unit vector. 
For small $t > 0,$ we have the bound
\ba %\label{5.1}
 s + t \rho - c t^2 \le r(t)
\le s + t \rho - \frac{a t^2}{8}.
\ea 
The above formula implies that there is $\tau=\tau(s, y; \rho, \nu)$ such that
$r(\tau)=0$ and $r(t)>0$ for $t< \tau$.
In addition,  $\K  \subset \Gamma$  is closed and $\gamma$
 is unit speed.
Thus there is  $\sigma > 0$ such that, for $s \le \sigma,\, y \in \K, \, \rho \le \sigma$ with $(s, y) \in U_0$,
the geodesic $\gamma(t)$ intersects $\Gamma$  at $t=\tau$.
Moreover, $y(t) \in \Gamma$ for $0 \le t \le \tau$ and $\gamma(t)$ is the distance minimizing up to 
$z_\gamma=\gamma(\tau)$.
Thus, for $t \in (0,\tau)$,
\beq \label{5.3}
d(z_\gamma,\gamma(t)) = \tau -t< d(z_\gamma,\gamma(0)).
\eeq 
%Analysing (\ref{5.1}) further we see that there are $\delta >0$ and $\rho_0 \le \sigma,\, r_0 \le \sigma$ such that,
%for $ s <r_0,\,\, |\eta_0|_h^2=1-\rho_0^2$, 
%\beq \label{5.2}
%\tau(s,y; \rho_0, \eta_0) >\delta, \quad (\nu, \, \dot \gamma(\tau(s,y; \rho, \eta)) >\delta.
%\eeq 
%
Let us emphasize that the case $s = 0$ is also allowed in the above argument.
We take $U = \{(s,y) \in U_0;\, s < \sigma \}$
and $\rho_0 = \sigma$.
% and note that $\tau(s,y; \rho_0, \eta_0)$ is continuously dependent on $(s,y,\, \eta_0)$.  

Let $(s,y) \in U$, $(\rho, \eta)$ be a unit vector and suppose that $\rho > \rho_0$.
We may choose $\eta_0=b \eta, \, 0 < b<1,$ such that $(\rho_0, \eta_0)$ is also a unit vector at $(s,y)$.
 Then the geodesic $\gamma_0(s)$ with the initial data  
$$
\gamma_0(0) = (s, y), \quad \dot \gamma_0(0) = (\rho_0, \eta_0)
$$
intersects $\Gamma$ at $z_{\gamma_0}$ and is distance minimizing until the intersection. 

\begin{figure}
\def\svgwidth{9cm}
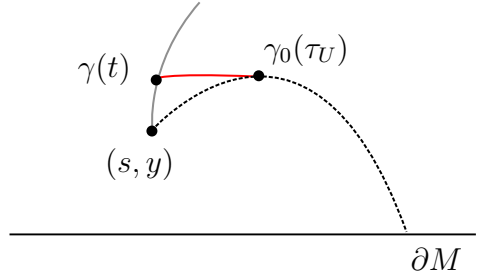
\caption{
A schematic of the short cut argument in the proof of Lemma \ref{lem_convexity}. The 
geodesics $\beta_t$, $\gamma$ and $\gamma_0$
are depicted by the
solid red, solid grey and dashed black curves, respectively.
}\label{fig_beta_shortcut}
\end{figure}
As $\rho_0 > 0$ we have that $\tau(s,y;\rho_0,\eta)$ is strictly positive 
for $(s,y) \in \overline U$ and $\eta_0 \in S := \{\eta \in \R^{n-1};\ 
|\eta|_h^2 + \rho_0^2 = 1\}$.
Together with continuity of $\tau$ this implies
$$
\tau_U := \min_{(s,y) \in \overline U, \eta_0 \in S}\tau(s,y; \rho_0, \eta_0)/2
> 0.
$$
Let $\beta_t : [0,l_t] \to M$ be the distance minimizing unit geodesic from $\gamma(t)$ to $\gamma_0(\tau_U)$,
see Figure \ref{fig_beta_shortcut}.
The first variation formula, see e.g. \cite[Prop. 10.2]{ONeill1983}, implies that 

$$
\p_t d(\gamma_0(\tau_U), \gamma(t))|_{t=0} = 
\p_t \left.\int_0^{l_t} |\dot \beta_t(r)| dr\right|_{t=0}
= -(\dot \beta_t(0), \p_t \beta_t(0))_g|_{t=0}.
$$
Observe that $\dot \beta_t(0)|_{t=0} = \dot \gamma_0(0)$
and that $\p_t \beta_t(0) = \dot\gamma(0)$. Hence
$$
\p_t d(\gamma_0(\tau_U), \gamma(t))|_{t=0} = -(\dot \gamma_0(0), \dot \gamma(0))_g = -\rho_0 \rho - b|\eta_0|_{h}^2 \le -\rho_0^2.
$$
It follows from the above inequality together with the relative compactness
of $U$ that there is $\delta > 0$ such that, if $t \in (0,\delta)$, $(s,y) \in U$, $\rho > \rho_0$,
then 
%then for small $s > 0$
\beq \label{5.4}
d(z_{\gamma_0}, \gamma(t)) \le d(z_{\gamma_0}, \gamma(0)) - t\rho_0^2/2.
\eeq 
The claim now follows from (\ref{5.3}) and (\ref{5.4})
by taking $\delta(\mathcal K) = \min(\tau, \delta)$, $x= \gamma(0)$ and $q = \gamma(t)$ for $t \in (0,\delta(\mathcal K))$.
\end{proof}

\begin{lemma}
\label{lem_far_away}
Let $\Gamma \subset \p M$ be open and let $x \in M_\Gamma$.
Then, for all $q \in M(\Gamma, d(x,\Gamma)) \setminus \{x\}$, there is $z \in \overline{\Gamma}$ satisfying $d(z,q) < d(z,x)$.
\end{lemma}
\begin{proof}
Let $x = (s,y),\,s=d(x, \Gamma),$ in  coordinates (\ref{boundary_normal_coords_s_y}),
and let $z$ be a closest point to $q$ in $\overline \Gamma$. 
If $z \ne y$ then 
$$
d(z, q) = d(q,\Gamma) \le d(x, \Gamma) < d(z, x),
$$
since $z$ is not the closest point to $x$ in $\Gamma$.
Suppose now that $z = y$ and write $r = d(y,q)$.
Then $r \le d(x,\Gamma) = s$ and $q = (r,y)$ in  coordinates (\ref{boundary_normal_coords_s_y}).
Moreover $q \ne x$, whence $r < s$.
\end{proof}

%We define $B_{\partial M}(y, \epsilon) = \{x \in \p M;\ d(x,y) < \epsilon \}$.

\begin{lemma}
\label{lem_cutoff_set_Gamma}
Let $\delta > 0$
and let $p \in M_\Gamma$
have the boundary normal coordinates $(s,y)$. 
Then there is $\epsilon = \epsilon(p, \delta) > 0$ such that
for all $x \in B(p, \epsilon)$,
\beq
\label{cutoff_set_Gamma}
M( B_{\partial M}(y, \epsilon), s+\epsilon) \subset M(\Gamma, d(x, \Gamma)) \cup B(x,\delta).
\eeq
\end{lemma}
\begin{proof}
To prove (\ref{cutoff_set_Gamma}) we assume the contrary. Then there exist sequences
$\epsilon_n \to 0$,
$$
x_n = (r_n, z_n) \in B(p,\epsilon_n), 
\quad
q_n \in M(B_{\partial M}(y, \epsilon_n), s+ \epsilon_n),
$$ 
such that 
$d(q_n, \Gamma) > r_n$ and $d(q_n, x_n) \ge \delta$.
Taking if necessary a subsequence, we may assume that 
$q_n \to q$. Then it follows from the above that
$$
d(q, y) \le s,\quad d(q, \Gamma) \ge s, \quad d(q, p) \ge\delta.
$$
This is a contradiction since the first two conditions imply $q=p$.
\end{proof}

Armed with these lemmas we are now in position to complete the proof of Proposition \ref{prop_convergence_to_pt}.

\begin{proof}[Proof of Proposition \ref{prop_convergence_to_pt}]
We assume that $\delta  > 0$ and $U  \subset M_\Gamma$ are as in Lemma \ref{lem_convexity}. Moreover, for $p = (s,y) \in U$ we fix  $\epsilon = \epsilon(p, \delta) > 0$  as in Lemma \ref{lem_cutoff_set_Gamma}. We decrease $\epsilon > 0$, if necessary,
so that $B_{\p M}(y, \epsilon) \subset \Gamma$ and 
that 
%we choose $\epsilon'\in(0,\epsilon)$ such that, 
in the coordinates (\ref{boundary_normal_coords_s_y}), 
$C(p,\epsilon) \subset B(p,\delta)$. Then, for any $x\in C(p,\epsilon)$ we set
\begin{equation}\label{def_hkx}
h_{k,x}(z)=d(z,x) - 1/k,\quad  z \in \Gamma.
\end{equation}
The set $X_{k,x}$ is visualized in Figure \ref{fig_Mh}. 
It is clear that $X_{k+1,x} \subset X_{k,x}$ 
and that $x \in X_{k,x}$ for all $k\in\mathbb N$. 
Suppose that $q \in \overline X_{k,x}$ for all $k\in\mathbb N$.
If $q \notin B(x,\delta)$ then (\ref{cutoff_set_Gamma}) yields that  $q \in M(\Gamma,d(x,\Gamma))$. %\setminus \{p\}$.
Now Lemma \ref{lem_far_away} implies that $q \in M(\Gamma,h_{k,x})^\inter$ for large $k$
which is a contradiction with $q \in \overline X_{k,x}$.
If, however, $q \in B(x,\delta) \setminus \{x\}$, then Lemma \ref{lem_convexity} implies that $q \in M(\Gamma,h_{k,x})^\inter$
 for large $k$
which is again a contradiction.
Thus $q = x$.
As the sequence of sets $X_{k,x}$ is decreasing and $\underset{k \ge 1}{\bigcap} \bar{X_{k,x}} = \{x\}$,
we have that $\diam(X_{k,x}) \to 0$ as $k \to \infty$.
\end{proof}

%\begin{lemma}
%\label{lem_diam}
%Let $p \in M$ and let $\Omega_j \subset M$, $j = 1, 2, \dots$.
%Suppose that $\Omega_{j+1} \subset \Omega_j$ and that $\bigcap_{j \ge 1} \overline{\Omega_j} = \{p\}$.
%Then $\diam(\Omega_j) \to 0$ as $j \to 0$.
%\end{lemma}
%{\footnotesize
%\begin{proof}
%To get a contradiction, suppose that $\diam(\Omega_j) \not\to 0$ as $j \to \infty$.
%Then there is $\epsilon > 0$ and $q_j^1, q_j^2 \in \Omega_j$ such that $d(q_j^1, q_j^2) \ge 25\epsilon$.
%We have $d(q_j^k, p) \ge \epsilon$ for $k=1$ or $k=2$ since otherwise $d(q_j^1, q_j^2) < 2\epsilon$
%by the triangle inequality. 
%By compactness of $M$ there is a converging sequence $(q_j)_{j=1}^\infty \subset M$
%such that $q_j \in \Omega_j$ and $d(q_j, p) \ge \epsilon$. Let us denote the limit by $q$.
%Then $d(q, p) \ge \epsilon$ and $q \ne p$.
%But $q \in \overline{\Omega_j}$ for all $j$, whence $q = p$ which is a contradiction.
%\end{proof}
%}

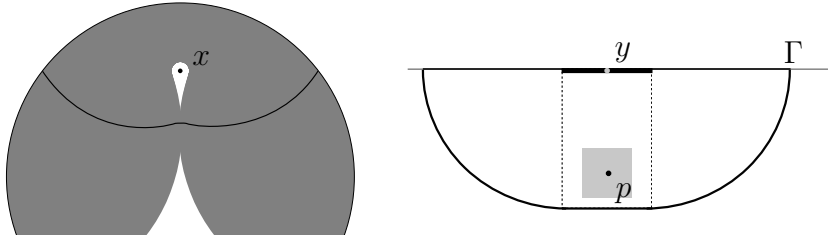
\begin{figure}[t]
\begin{subfigure}[t]{5cm}
\centering
\setlength{\unitlength}{5cm}
\begin{picture}(1,0.64)%
\put(0,0){\includegraphics[clip=true, trim=0 3cm 0 0, width=5cm]{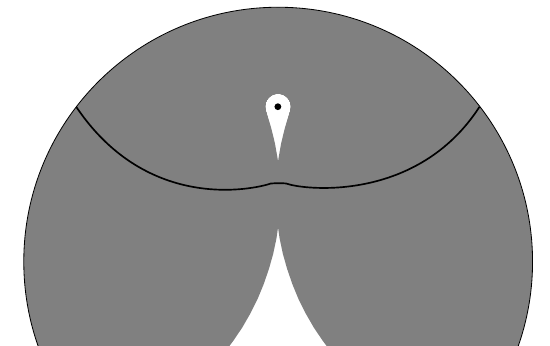}}%
\put(0.53778238,0.45){\color[rgb]{0,0,0}\makebox(0,0)[lb]{\smash{$x$}}}%
\end{picture}%
\end{subfigure}
\quad 
\begin{subfigure}[t]{5cm} 
\centering
\vspace{-2.6cm}
\def\svgwidth{6cm}
\begingroup%
  \makeatletter%
  \providecommand\color[2][]{%
    \errmessage{(Inkscape) Color is used for the text in Inkscape, but the package 'color.sty' is not loaded}%
    \renewcommand\color[2][]{}%
  }%
  \providecommand\transparent[1]{%
    \errmessage{(Inkscape) Transparency is used (non-zero) for the text in Inkscape, but the package 'transparent.sty' is not loaded}%
    \renewcommand\transparent[1]{}%
  }%
  \providecommand\rotatebox[2]{#2}%
  \ifx\svgwidth\undefined%
    \setlength{\unitlength}{160.55349121bp}%
    \ifx\svgscale\undefined%
      \relax%
    \else%
      \setlength{\unitlength}{\unitlength * \real{\svgscale}}%
    \fi%
  \else%
    \setlength{\unitlength}{\svgwidth}%
  \fi%
  \global\let\svgwidth\undefined%
  \global\let\svgscale\undefined%
  \makeatother%
  \begin{picture}(1,0.37874767)%
    \put(0,0){\includegraphics[width=\unitlength]{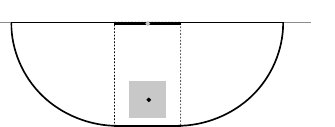}}%
    \put(0.83336453,0.3249619){\color[rgb]{0,0,0}\makebox(0,0)[lb]{\smash{$\Gamma$}}}%
    \put(0.46140047,0.02738486){\color[rgb]{0,0,0}\makebox(0,0)[lb]{\smash{$p$}}}%
    \put(0.45928582,0.34089035){\color[rgb]{0,0,0}\makebox(0,0)[lb]{\smash{$y$}}}%
  \end{picture}%
\endgroup%
\end{subfigure}
\caption{
{\em Left.} A part of the domain of influence $M(\Gamma, h_{k,x})$ in gray,
where $(M,g)$ is the Euclidean unit disk, $x=(0,3/5)$, $k=20$,
and $\Gamma$ is slightly less than the upper half circle.
The black curve is the boundary of $M(B_{\p M}(y, \epsilon), s+\epsilon)$
where $p=(s,y)$, $y=(0,1)$, $s=1/2$ and $\epsilon = 1/5$.
The set $X_{k,x}$ is the white region around $x$.
{\em Right.} Schematic diagram of the sets $B_{\p M}(y, \epsilon) \subset \Gamma$, in black around the gray point $y$, and $C(p,\epsilon')$, in gray around the black point $p=(s,y)$. Here $\epsilon > \epsilon' > 0$. The black curve is the boundary of $M(B_{\p M}(y, \epsilon), s + \epsilon$).}
\label{fig_Mh}
\end{figure}

\subsection{Localized solutions}
\label{sec_local_rec}

We denote by $|X|$ the Riemannian volume of a measurable set $X \subset M$. Also, we write 
$$
\mathbb U_{A,q}: L^2((0,T) \times \Src)\ni f \to u_{f}(T,\cdot)\in L^2(M),
$$
where $u_f$ is the solution of (\ref{eq_wave}).
If (\ref{eq_wave}) is exactly controllable from $\Src$ in time $T$, then $\mathbb U_{A,q}$
is surjective.
In this case, by Banach-Schauder theorem,  $\mathbb U_{A,q}$ admits a pseudoinverse 
\ba
\mathbb U_{A,q}^\dagger : L^2(M) \to L^2((0,T) \times \Src)
\ea
which is continuous. Moreover, the composition $\mathbb U_{A,q} \mathbb U_{A,q}^\dagger$ gives the identity map, see e.g. \cite[pp. 33-34]{Engl1996}.

In this subsection we will prove the following proposition that, together with Proposition \ref{prop_convergence_to_pt}, will allow us to enforce a sequence of solutions of \eqref{eq_wave} to converge, at a fixed time, to a point mass.

\begin{proposition}
\label{prop_localization}
Let $\mathcal X \subset M$ be open, $x \in \mathcal X$ and let $X_k \subset M$, $k\in\mathbb N$, be a sequence of neighborhoods of $x$ satisfying 
$\underset{k \to \infty}{\lim}\diam(X_k) = 0$. % as $j \to \infty$.
Let $\psi_0 \in C_0^\infty(\mathcal X)$ satisfy $\psi_0(x) \ne 0$.
Let $T >0$ and suppose that a sequence $(f_k)_{k\in\mathbb N}$ of functions in $L^2((0,T) \times \Src)$ satisfies
\begin{itemize}
\item[(i)] there is $C > 0$ such that $\norm{f_k}_{L^2((0,T) \times \Src)} \le C |X_k|^{-1/2}$ for all $k\in\mathbb N$,
\item[(ii)] $\supp(u_{f_k}(T,\cdot)) \subset \bar X_k \cup (M \setminus \mathcal X)$
for all $k\in\mathbb N$,
\item[(iii)] $(\bra u_{f_k}(T,\cdot), \psi_0 \cet_{L^2(M)} )_{k\in\mathbb N}$ converges.
\end{itemize}
Then there is $\kappa \in \C$ such that $\bra u_{f_k}(T,\cdot), \psi \cet_{L^2(M)} \to \kappa \psi(x)$ for all functions $\psi \in C_0^\infty(\mathcal X)$.

Furthermore, if the wave equation (\ref{eq_wave}) is exactly controllable from $\Src$ in time $T$,
then the sequence $f_k = \mathbb U_{A,q}^\dagger 1_{X_k} / |X_k|$, $k\in\mathbb N$, satisfies (i)-(iii), and the corresponding $\kappa$ is $1$.
\end{proposition}
Let us emphasize that $\mathcal X$ is open in the topology of $M$, a manifold with boundary. In particular, $\mathcal X$ may intersect $\p M$ in which case $x$ may belong to $\p M$.
\begin{proof}
Let $\psi \in C_0^\infty(\mathcal X)$. Observe that $\supp(u_{f_k}(T,\cdot) \psi) \subset \bar X_k$. Fixing $$R_k(\psi) =\bra u_{f_k}(T,\cdot), \psi\cet_{L^2(M)}-\psi(x)\int_{X_k}\overline{u_{f_k}(T,x)}dV_g(x),$$ 
where $dV_g$ denotes the Riemannian volume, we get
%We may perform the below computations in coordinates around $q$.
%We have
$$
\bra u_{f_k}(T,\cdot), \psi\cet_{L^2(M)} = \psi(x) \bra u_{f_k}(T,\cdot), 1\cet_{L^2(X_k)} + R_k(\psi).
$$

Using some local coordinates $\tilde x$ in $X_k$
for all large enough $k$,
 the remainder term satisfies
\ba
|R_k(\psi)| &\le&  \int_{X_k} |u_{f_k}(T,\tilde{x})| |\psi(\tilde{x})-\psi(x)| dV_g(\tilde x)\\
 &\le& C \norm{d \psi}_{C(X_k)} \int_{X_k} |u_{f_k}(T,\tilde{x})| d(\tilde x,x) dV_g(\tilde x)
\\&\le  & C \norm{d \psi}_{C(X_k)} ||u_{f_k}(T,\cdot)||_{L^2(M)} \left( \int_{X_k} d^2(\tilde x,x) dV_g(\tilde x)\right)^{1/2}
\\&\le  & C \norm{d\psi}_{C(M)} \diam(X_k) \to 0.
\ea
Notice that the constant $C > 0$ may increase between the inequalities and that, at the last inequality, we use $$||u_{f_k}(T,\cdot)||_{L^2(M)} \le C \norm{f_k}_{L^2((0,T) \times \Src)},$$
see \cite{Lasiecka1986}, together with (i).
We choose $\psi = \psi_0$ and see that the limit 
$$
\underset{k \to \infty}{\lim} \bra u_{f_k}(T,\cdot), 1\cet_{L^2(X_k)} 
= 
\frac{1}{\psi_0(x)} 
\underset{k \to \infty}{\lim} (\bra u_{f_k}(T,\cdot), \psi_0\cet_{L^2(M)} - R_k(\psi_0))
$$
exists. We denote the limit by $\kappa$. Thus for any $\psi \in C_0^\infty(\mathcal X)$ it holds that $\bra u_{f_k}(T,\cdot), \psi\cet_{L^2(M)} \to \kappa \psi(x)$ as $k \to \infty$.

Finally, it is clear that $f_k = \mathbb U_{A,q}^\dagger 1_{X_k} / |X_k|$ has the properties (i)-(iii) with $\kappa =1$.
\end{proof}

\subsection{Local recovery near the set $\Rec$}
\def\CC{\mathcal C}

Armed with the localization procedure given by Propositions \ref{prop_convergence_to_pt} and  \ref{prop_localization}, we prove Theorem \ref{th_main} in this section.

From now on, we fix $A_j\in  C^\infty(M;TM)$, $q_j\in\mathcal C^\infty(M;\mathbb C)$, $j=1,2$, and, for functions $f\in C^\infty_0((0,+\infty)\times \p M)$, $\phi\in C^\infty_0((0,+\infty)\times \p M)$ and $H\in C^\infty_0((0,+\infty)\times B)$, we consider $u_{j,f}$, $v_{j,\phi}$, $w_{j,H}$ solving respectively  \eqref{eq_wave}, \eqref{eq_wave_adjoint1}, \eqref{eq_wave_adjoint2} with $A=A_j$ and $q=q_j$.
We write also $$\mathcal A_j = \Delta_g - A_j - q_j,\quad \mathcal A_j^* = \Delta_g + \overline{A_j} - (\overline{q_j}-\textrm{div}_g(\overline{A_j})).$$

Before proving Theorem \ref{th_main} we still need to establish two lemmas.

\begin{lemma}
\label{lem_recovery_Ak}
Let $\Gamma \subset \p M$ and $B \subset M^\inter$ be open.
Let $T >0$ and $h : \bar \Gamma \to [0, T]$ be piecewise continuous.
Let $\CC \subset M(\Gamma, h) \cap M^\inter$ be open and let $\kappa \in C^\infty(\CC)$
be nowhere vanishing.
Then the condition 
\begin{equation}\label{l10a}
 v_{1,\phi}(T,x)=\kappa v_{2,\phi}(T,x), \quad \phi \in C_0^\infty(\B(\Gamma, h; T)),\ x\in\CC,
\end{equation}
implies that $\mathcal A_{1}=\overline{\kappa}^{-1} \mathcal A_2 \overline{\kappa}$ on $\CC$.
In the same way, for $h : \bar \Gamma \to [0, T]$ piecewise continuous,
$\CC \subset M(B, h) \cap M^\inter$  and  $\kappa \in C^\infty(\CC)$
be nowhere vanishing, the condition 
\begin{equation}\label{l10b}
 w_{1,\kappa H}(T,x)=\kappa w_{2,H}(T,x), \quad H \in C_0^\infty(\B(B, h; T)),\ x\in\CC,
\end{equation}
implies that $\mathcal A_{1}=\overline{\kappa}^{-1} \mathcal A_2 \overline{\kappa}$ on $\CC$.
\end{lemma}
\begin{proof} Since the proof of these two results are similar, we will only show that \eqref{l10a} implies $\mathcal A_{1}=\overline{\kappa}^{-1} \mathcal A_2 \overline{\kappa}$ on $\CC$. We start by proving that \eqref{l10a} implies
\begin{equation}\label{l10c}   \mathcal A_1^*v_{1,\phi}(T,\cdot)=\kappa \mathcal A_2^* v_{2,\phi}(T,\cdot), \quad \phi \in C_0^\infty(\B(\Gamma, h; T)).
\end{equation}
For this purpose, we fix $\phi \in C_0^\infty(\B(\Gamma, h; T))$ and remark that there is $\varepsilon>0$ such that supp$(\phi)\subset [\epsilon,+\infty)\times \p M$ and $\phi\in C_0^\infty(\B(\Gamma, h-\epsilon; T))$. Thus, 
taking into account the translation invariance in time of (\ref{eq_wave_adjoint1}) and fixing $\phi_s:(t,x)\mapsto \phi(s+t,x)$, we obtain that $$v_{j,\phi_s}(T,\cdot)=v_{j,\phi}(s+T,\cdot), \quad s\in[0,\epsilon),\ j=1,2$$ and \eqref{l10a} implies
$$ v_{1,\phi}(s+T,\cdot)= \kappa v_{2,\phi}(s+T,\cdot), \quad s\in[0,\epsilon).$$
Differentiating twice this identity with respect to $s$, we get \eqref{l10c}.

Now let $\psi$ be a function in $C_0^\infty(\CC)$. 
Applying again \eqref{l10a},
we can compute 
$$\begin{aligned}
\bra \kappa \mathcal A_2^* v_{2,\phi}(T,\cdot), \psi\cet _{L^2(\CC)}
&= \bra \kappa v_{2,\phi}(T,\cdot), \overline{\kappa}^{-1} \mathcal A_2 \overline{\kappa} \psi\cet_{L^2(\CC)}\\
\ &=\bra  v_{1,\phi}(T,\cdot), \overline{\kappa}^{-1} \mathcal A_2 \overline{\kappa} \psi\cet_{L^2(\CC)}.\end{aligned}
$$
Applying \eqref{l10c}, we get
$$\bra  \mathcal A_1^* v_{1,\phi}(T,\cdot), \psi\cet _{L^2(\CC)}=\bra \kappa \mathcal A_2^* v_{2,\phi}(T,\cdot), \psi\cet _{L^2(\CC)}=\bra  v_{1,\phi}(T,\cdot), \overline{\kappa}^{-1} \mathcal A_2 \overline{\kappa} \psi\cet_{L^2(\CC)}$$
and it follows
\begin{equation}\label{l10e}\bra v_{1,\phi}(T,\cdot), (\mathcal A_1  -\overline{\kappa}^{-1} \mathcal A_2 \overline{\kappa}) \psi\cet_{L^2(\CC)}=0,\quad \psi\in \mathcal C^\infty_0(\CC).\end{equation}
As the functions $v_{1,\phi}(T,\cdot)|_\CC$, $\phi \in C_0^\infty(\B(\Gamma, h; T))$, 
are dense on $L^2(\CC)$, we deduce from \eqref{l10e} that $\mathcal A_1 =\overline{\kappa}^{-1} \mathcal A_2 \overline{\kappa}$ on $\CC$, which completes the proof of the lemma.
\end{proof}

The next lemma will be used only for $j=2$. 

\begin{lemma}
\label{lem_enforce_smoothness}
Let $\Gamma \subset \p M$ and $B \subset M^\inter$ be open, and let $\Rem = \Gamma$ or $\Rem = B$.
Let $T > 0$ and let $h : \bar \Rem \to [0, T]$ be piecewise continuous.
In the case when
$\Rem = B$ suppose, moreover, that $h > 0$ pointwise.
Let $\CC_1 \subset M(\Gamma, h) \cap M^\inter$ and $\CC_2 \subset M(B, h) \cap M^\inter$ be open
and let $\kappa_\ell : \CC_\ell \to \C$. Then the following properties hold:
\begin{itemize}
\item[(1)] For $j=1,2$, if $\kappa_1 v_{j,\phi}(T,\cdot) \in C^\infty(\CC_1)$ for all $\phi \in C_0^\infty(\B(\Gamma, h; T))$
then $\kappa_1 \in C^\infty(\CC_1)$. In the same way, if $\kappa_2 w_{j,H}(T,\cdot)  \in C^\infty(\CC_2)$ for all $H \in C_0^\infty(\B(B, h; T))$ then $\kappa_2 \in C^\infty(\CC_2)$.
\item[(2)] If for all $x \in \CC_1$ there is $\phi \in C_0^\infty(\B(\Gamma, h; T))$ 
such that $$\kappa_1(x) v_{j,\phi}(T,x) \ne 0$$ then $\kappa_1(x) \ne 0$ for all $x \in \CC_1$.  If  for all $x \in \CC_2$ there is $H \in C_0^\infty(\B(B, h; T))$ 
such that $\kappa_2(x) w_{j,H}(T,x) \ne 0$ then $\kappa_2(x) \ne 0$ for all $x \in \CC_2$.
\end{itemize}

Moreover, in the case $\Rem = \Gamma$ we can enforce smoothness up to the boundary, that is, 
we define $\tilde \CC = \CC_1 \cup (S \cap \bar \CC_1)$
where $S$ is an open set in $\p M$ such that $h > 0$
in $S$, and have the following:
\begin{itemize}
\item[(3)] For $j=1,2$, if $\kappa_1 v_{j,\phi}(T,\cdot) \in C^\infty(\tilde\CC)$ for all $\phi \in C_0^\infty(\B(\Gamma, h; T))$
then $\kappa_1 \in C^\infty(\tilde \CC)$.
\end{itemize}
\end{lemma}
\begin{proof} The proof for $\Rem = \Gamma$ or $\Rem = B$ being similar, we consider only the results for $\Rem=\Gamma$.
Let $x \in \CC_1$. By Lemma \ref{lem_non_vanishing}, there is a neighborhood $U$ of $x$
and $\phi \in C_0^\infty(\B(\Gamma, h; T))$
such that $v_{j,\phi}(T,\cdot)$ is non-vanishing in $U$.
We have that $\kappa_1 = \frac{\kappa_1v_{j,\phi}(T,\cdot)}{v_{j,\phi}(T,\cdot)} $ in $U$,
and (1) and (2) follow.

%Let $x \in \tilde \CC \setminus \CC$.
%If $\Rem = B$, 
%then approximate controllability implies again that 
%there is a neighborhood $U \subset M^\inter$ of $x$
%and $\phi \in C_0^\infty(\B(\Rem, h))$
%such that $W_\Rem \phi$ is non-vanishing in $U$.
%Thus $\mu$ is smooth in $U$.

Suppose now that  $x \in S \cap \bar \CC_1$.
Then there is a neighborhood of $U \subset M$ of $x$
and $\phi \in C_0^\infty(\B(\Gamma, h; T))$
such that $v_{j,\phi}(T,\cdot)$ is non-vanishing in $U$,
since on $S$ we can choose $v_{j,\phi}(T,\cdot) = \phi(T,\cdot)$ to be non-vanishing. 
The function $v_{j,\phi}(T,\cdot)$ is smooth up to $\p M$, %, see e.g. [TODO Evans, Th. 7.2.6].
whence $\kappa_1$ is smooth in $U$.
\end{proof}

We are ready to prove the local result formulated in the introduction.

\begin{proof}[Proof of Theorem \ref{th_main}]

As the wave equation is exactly controllable from $\Src$ in time $T$,
\cite[Theorem 3.2]{Bardos1992}
implies that $\sigma_\Rec \le T$ pointwise on $\Rec$.
We recall that $\sigma_\Rec$ is defined by (\ref{def_sigma}). For $j=1,2$ and $\kappa\in C^\infty(\overline{\Omega_s})$, we fix $K_{j,\Rec}^{T}$ given by \eqref{def_K_Gamma} for $\Lambda_{\Gamma}=\Lambda^{2T}_{j,S,\Rec}$ and applying \eqref{t1a} we deduce that
\begin{equation}\label{th1a} K_{1,\Rec}^{T}=K_{2,\Rec}^{T}.\end{equation}
Let $\mathcal K \subset \Rec$ be compact,
and consider the sets defined in Proposition \ref{prop_convergence_to_pt} for $\Gamma=\Rec$, $U\subset M_\Rec$ and $p=(s,y)\in U$.
We write $C_p(\mathcal K) = C_p$ to emphasize the dependence on $\mathcal K$, and use an analogous notation also for other quantities in Proposition \ref{prop_convergence_to_pt}.
 We will start by proving that there exists $\kappa\in  C^\infty(C_p(\mathcal K))$ such that the following identity holds
\begin{equation}\label{th1}  v_{1,\phi}(T,x)=\kappa(x) v_{2,\phi}(T,x),\end{equation}
for $\phi\in C^\infty_0(\mathcal B(B_{\partial M}(y,\epsilon), s+\epsilon;T ))$ and $x\in C_p(\mathcal K)$.
For any $x \in C_p(\K)$, we consider functions $h_{k,x}$ and sets $X_{k,x}$ satisfying the properties described in Proposition \ref{prop_convergence_to_pt}.
We will apply the result of Proposition \ref{prop_localization}, with $\mathcal X = M(B_{\partial M}(y,\epsilon), s+\epsilon)^\inter$ and $X_k=X_{k,x}$. Using the exact controllability assumption, we fix 
    \begin{equation}\label{f_pseudoinv}
f_{k,x} = \mathbb U_{A_1,q_1}^\dagger 1_{X_{k,x}} / |X_{k,x}|,\quad k\in\mathbb N.
    \end{equation}
We remark that
%, for $j=1$, the conditions (i)-(iii) of Lemma \ref{lem_localization} are fulfilled and we have 
\begin{equation}\label{th2} \lim_{k\to+\infty} \left\langle u_{1,f_{k,x}}(T,\cdot), \psi \right\rangle_{L^2(M)}=\psi(x),\quad \psi\in C^\infty_0(\mathcal X ).\end{equation}
Let us now show that the conditions (i)-(iii) of Proposition \ref{prop_localization} are fulfilled with respect to $j=2$.
Clearly (i) holds, as it does not depend on $j=1,2$. 
The equations \eqref{blago_inner_prod1} and \eqref{th1a}
imply that 
\begin{equation}\label{th2a} \begin{aligned}\left\langle u_{1,f_{k,x}}(T,\cdot), v_{1,\phi}(T,\cdot) \right\rangle_{L^2(M)}&=\bra K_{1,\Rec}^{T} f_{k,x},\phi \cet_{L^2((0,T) \times \Rec)}\\
\ &=\bra K_{2,\Rec}^{T} f_{k,x},\phi \cet_{L^2((0,T) \times \Rec)}\\
\ &=\left\langle u_{2,f_{k,x}}(T,\cdot), v_{2,\phi}(T,\cdot) \right\rangle_{L^2(M)}.\end{aligned}\end{equation}
Finite speed of propagation implies that $\supp(v_{1, \phi}(T,\cdot)) \subset M(\Gamma, h_{k, x})$ for all
$\phi \in C^{\infty}_0 (\B(\Rec, h_{k,x}; T))$.
Observe that $u_{1,f_{k,x}}(T,\cdot) = 1_{X_{k,x}} / |X_{k,x}|$ by (\ref{f_pseudoinv}), and recall that by definition $\mathcal X = M(B_{\partial M}(y,\epsilon), s+\epsilon)^\inter$ and $X_{k,x} = \mathcal X \setminus M(\Gamma, h_{k, x})$, see Proposition \ref{prop_convergence_to_pt} for the latter.
Hence
$$\left\langle u_{2,f_{k,x}}(T,\cdot), v_{2,\phi}(T,\cdot) \right\rangle_{L^2(M)}=\left\langle u_{1,f_{k,x}}(T,\cdot), v_{1,\phi}(T,\cdot) \right\rangle_{L^2(M)}=0$$
and the density result of Lemma \ref{Lem. Tataru} implies that 
$u_{2,f_{k,x}}(T,\cdot) = 0$ in $M(\Gamma, h_{k, x})$.
Therefore 
$$
\supp(u_{2,f_{k,x}}(T,\cdot)) \subset 
M \setminus M(\Gamma, h_{k, x}) \subset X_{k,x} \cup (M \setminus \mathcal X),
$$
and (ii) holds.
Moreover, equations \eqref{th2}-\eqref{th2a} imply that the sequence $$(\left\langle u_{2,f_{k,x}}(T,\cdot), v_{2,\phi}(T,\cdot) \right\rangle_{L^2(M)})_{k\in\mathbb N}$$ converges
for any $\phi\in C^\infty_0(\B(B_{\partial M}(y,\epsilon), s+\epsilon;T ))$. Thus, condition (iii) of Proposition \ref{prop_localization} is also fulfilled. Note that by Lemma \ref{lem_non_vanishing}, the function $\phi$ can be chosen so that $v_{2,\phi}(T,x) \ne 0$. According to Proposition \ref{prop_localization} there exists $\kappa_x$ such that
\begin{equation}\label{th2b}\lim_{k\to+\infty}\left\langle u_{2,f_{k,x}}(T,\cdot), \psi \right\rangle_{L^2(M)}=\kappa_x\psi(x),\quad \psi\in C^\infty_0(\mathcal X).\end{equation}

We define a function $\kappa: C_p(\K) \to \C$ 
by $\kappa(x) = \kappa_x$, and remark that applying \eqref{th2} with $\psi=v_{1,\phi}(T,\cdot)$ for $\phi\in C^\infty_0((B_{\partial M}(y,\epsilon), s+\epsilon,T ))$, and using \eqref{th2a}, we have
\begin{equation}\label{tata}
v_{1,\phi}(T,x) = \lim_{k\to+\infty}\left\langle u_{2,f_{k,x}}(T,\cdot), v_{2,\phi}(T,\cdot) \right\rangle_{L^2(M)}=\kappa(x) v_{2,\phi}(T,x).
\end{equation}
This establishes \eqref{th1}. Moreover,  combining \eqref{th1} with the fact that $v_{1,\phi}(T,\cdot)\in C^\infty(M)$, for $\phi\in C^\infty_0(\mathcal B(B_{\partial M}(y,\epsilon), s+\epsilon,T ))$, and applying  Lemma \ref{lem_enforce_smoothness}, we deduce that $\kappa\in C^\infty(C_p(\mathcal K)\cup(\overline{C_p(\mathcal K)}\cap \Rec) )$. Finally, applying Lemma \ref{lem_non_vanishing} we deduce that for all $x\in C_p(\mathcal K)$ there exists $\phi_x$ such that $v_{\phi_x}(T,x)\neq0$. Thus, \eqref{tata} and Lemma \ref{lem_enforce_smoothness} imply $\kappa$ is nowhere vanishing in $C_p(\mathcal K)$.

We consider now  a collection $\{\K_i;\ i\in I\}$ of compact sets in $\Rec$ and $p_i=(s_i,y_i)$ lying in a neighborhood of $\K_i$ as in Proposition \ref{prop_convergence_to_pt}. We assume that $\{\K_i;\ i\in I\}$ and  $\{p_i;\ i\in I\}$ are chosen in such a way that $\bigcup_{i\in I} C_{p_i}(\K_i) $ is a neighborhood of $\Rec$. Repeating the above argumentation, for all $i\in I$,  we find $\kappa_i\in C^\infty(C_{p_i}(\K_i)\cup(\overline{C_{p_i}(\mathcal K_i)}\cap \Rec))$ such that, for all $x\in C_{p_i}(\K_i)$, we have
\begin{equation}\label{th2c} v_{1,\phi}(T,x)=\kappa_i(x) v_{2,\phi}(T,x),\quad \phi\in C^\infty_0(B_{\partial M}(y_i,\epsilon_i), s_i+\epsilon_i,T )).\end{equation}
Now let $i_1,i_2\in I$ be such that $C_{p_{i_1}}(\K_{i_1})\cap C_{p_{i_2}}(\K_{i_2})\neq\emptyset$. Since both $C_{p_{i_1}}(\K_{i_1})$ and $ C_{p_{i_2}}(\K_{i_2})$ are cylindrical domains, $C_{p_{i_1}}(\K_{i_1})\cap C_{p_{i_2}}(\K_{i_2})$ is also cylindrical and we have $B_{\partial M}(y_{i_1},\epsilon_{i_1})\cap B_{\partial M}(y_{i_2},\epsilon_{i_2})\neq\emptyset$. Combining this with \eqref{th2c},  we obtain for any $$\phi\in C^\infty_0(\B(B_{\partial M}(y_{i_1},\epsilon_{i_1})\cap B_{\partial M}(y_{i_2},\epsilon_{i_2}), \max(s_{i_1}+\epsilon_{i_1},s_{i_2}+\epsilon_{i_2}),T ))$$
and any $x\in C_{p_{i_1}}(\K_{i_1})\cap C_{p_{i_2}}(\K_{i_2})$, the equation
\begin{equation}\label{th2d} \kappa_{i_2}(x) v_{2,\phi}(T,x)=v_{1,\phi}(T,x)=\kappa_{i_1}(x) v_{2,\phi}(T,x).\end{equation}
In view of Lemma \ref{lem_non_vanishing},  for any $x\in C_{p_{i_1}}(\K_{i_1})\cap C_{p_{i_2}}(\K_{i_2})$, we can choose $\phi$ such that $v_{2,\phi}(T,x)\neq0$. Combining this with \eqref{th2d}, we have
\begin{equation}\label{th2e} \kappa_{i_1}(x)=\kappa_{i_2}(x) ,\quad x\in C_{p_{i_1}}(\K_{i_1})\cap C_{p_{i_2}}(\K_{i_2}).\end{equation}
Therefore, we can define $\kappa\in C^\infty\left(\bigcup_{i\in I} C_{p_i}(\K_i)\right)$ such that, for all $i\in I$, $\kappa_{|C_{p_i}(\K_i)}=\kappa_i$. In light of \eqref{th2c}, we deduce that $\kappa_{|\Rec}=1$ and  Lemma \ref{lem_enforce_smoothness} implies that $\kappa\neq0$. Moreover,  applying Lemma \ref{lem_recovery_Ak} on $C_{p_i}(\K_i)$, for any $i\in I$, we deduce that 
$$\mathcal A_{1}=\kappa^{-1} \mathcal A_2 \kappa=\Delta_g - (A_2 +2\kappa^{-1}\textrm{grad}_g\kappa) - (q_2+\kappa(A_2-\Delta_g)\kappa^{-1})$$
holds true on $\bigcup_{i\in I} C_{p_i}(\K_i) $. Therefore, we can define $\mathcal U$ a neighborhood of $\Rec$, contained into $\bigcup_{i\in I} C_{p_i}(\K_i) $, such that  $\kappa\in C^\infty(\overline{\mathcal U})$ and  $(A_1|_{\mathcal U},q_1|_{\mathcal U})\in \mathcal G_{\mathcal U, \Rec}(A_2,q_2)$. This completes the proof of the theorem.
\end{proof}

\section{Reconstruction 
 of the first order perturbation
 along a convex foliation}
\label{sec_global}

In this section we prove the global result stated in Theorem \ref{th_main_foliation}.
The proof of this result is based on iterating the local reconstruction method of the previous section along the convex foliation. 

\subsection{Local recovery near the set $\Sigma_s$}

Let $\Sigma_s$, $s \in (0, 1]$, be a convex foliation satisfying (F1)-(F7).
Let $\Gamma \subset \Sigma_s$ be open and let $h : \bar \Gamma \to \R$ 
be piecewise continuous. We recall that $M_s$ is 
 defined in (F4),
and consider the domain of influence on $M_s$,
\ba
& &M_s(\Gamma, h)  := \{x \in M_s;\ \inf_{y \in \Gamma} (d_{M_s}(x, y) - h(y)) \le 0\}.
\ea
Here $d_{M_s}(x, y)$ is the distance function on $(M_s, g)$.
We will also use the notation $d_{\overline \Omega_s}(x,y)$ for the distance function on $(\overline \Omega_s, g)$.

\begin{lemma}
\label{lem_Ms}
Let $\Sigma_s$, $s \in (0, 1]$, be a convex foliation satisfying (F1)-(F7),
and let $s \in (0, 1]$. Let $h : \bar \Sigma_s \to \R$ 
be piecewise continuous.
Then
\begin{equation}
\label{eq_Ms}
M_s(\Sigma_s, h)\cup \overline \Omega_s=M (\Omega_s, \tilde h),
\end{equation}
where $\tilde h(y)=\max(\sup_{z\in \Sigma_s}(h(z)-d_{\overline \Omega_s}(z,y)),
d_{\overline \Omega_s}(y,\p \Omega_s))
$.	
\end{lemma}
\begin{proof}
Let us show first that 
$$
d(x,z) = d_{M_s}(x,z), \quad x,z \in M_s.
$$
It is enough to show that a shortest path $\gamma$ between $x$ and $z$ stays in $M_s$.
To get a contradiction suppose that $S < s$, where
$$
S = \inf\{r \in [0,s];\ \gamma \cap \Sigma_r \ne \emptyset \},
$$
and we have used the notation $\Sigma_0 = \Rec_0$.
Let  $p \in \gamma \cap \Sigma_S$.
Let us consider first the case $S > 0$.
Then $\gamma$ is a geodesic near $p$.
As $\gamma \cap \Omega_S = \emptyset$, the intersection is tangential.
But then the strict convexity of $\Sigma_S$ implies that $\gamma$ is in $\Omega_S$ near $p$,
which is a contradiction. 
On the other hand, if $S = 0$ then the intersection must be tangential again, 
since a shortest path is $C^1$, see \cite{AA}.
But this is impossible by the strict convexity of $\Sigma_0 \subset \Rec$. 

Let us now show (\ref{eq_Ms}).
Note that $\tilde h(y) \ge h(y)$ for $y \in \Sigma_s$
and that $\tilde h > 0$ on $\Omega_s$.
Hence $M_s(\Sigma_s, h)\cup \overline \Omega_s \subset M (\Omega_s, \tilde h)$. 
On the other hand, if $x \in M(\Omega_s, \tilde h) \setminus \bar \Omega_s$ then 
there is $y \in \bar \Omega_s$ such that $d(x,y) - \tilde h(y) \le 0$
and $z \in \bar \Sigma_s$ such that $\tilde h(y) = h(z) - d_{\overline \Omega_s}(z,y)$.
Thus 
$$
d_{M_s}(x,z) - h(z) = d(x,z) - d_{\overline \Omega_s}(z,y) - \tilde h(y)
\le d(x,y) - \tilde h(y) \le 0,
$$
and $x \in M_s(\Sigma_s, h)$.
\end{proof}

Let us prove next the following analogue of Theorem \ref{th_main} 
with internal data on $\Omega_s$. 
Note that contrary to Theorem \ref{th_main}, we do not require $\kappa$ to have a specific value on $\Sigma_s$.
We recall that for and open set $U \subset M^\inter$
and $f \in C_0^{\infty}((0, \infty) \times \Src)$,
$$\mathcal T_{j,U, \kappa} f = \kappa u_j|_{(0,\infty) \times U},$$
where $u_j$ is the solution of (\ref{eq_wave}) for $A=A_j$ and $q=q_j$.

\begin{lemma}
\label{lem_lot_Us}
Let $\Src \subset \p M$ be open and suppose that the wave equation (\ref{eq_wave}) is exactly controllable from $\Src$ in time $T > 0$.
Let $\Sigma_s$, $s \in (0, 1]$, be a convex foliation satisfying (F1)-(F7),
let $s \in (0,1]$, and let $\kappa_0 \in C^\infty(\overline{\Omega_s})$
be nowhere vanishing.
Then there is a neighborhood $\mathcal U_s \subset M_s$ of $\Sigma_s$ 
such that the condition
\begin{equation}
\label{llu1}\mathcal T_{1,\Omega_s, 1}=\mathcal T_{2,\Omega_s, \kappa_0}\end{equation}
implies that there exists $\kappa \in C^\infty(\mathcal U_s)$ such that 
$$\kappa(x) \ne 0, x \in \mathcal U_s$$
and 
\begin{equation}
\label{llu2}
\mathcal A_1|_{\mathcal U_s} =\kappa^{-1} \mathcal A_2\kappa|_{\mathcal U_s}.
\end{equation}

\end{lemma}  
\begin{proof}

%Then $(M_{\Sigma_s}, g)$ can be reconstructed from the data $\Lambda_{\Omega_s}$, see section \ref{sec_g_near_Sigma},
%and we can explicitly construct the functions $h_j$ and the sets $X_j$
%for any $x \in C_p$ and $p \in U(\K)$.
%The validity of all the conditions (i)-(iii)  of Lemma \ref{lem_localization}
%can be determined by using the data $\Lambda_{\Omega_s}$.

For $j=1,2$, we fix $K_{j,\Omega_s,\kappa}$ given by \eqref{def_K_B} for $B=\Omega_s$  and $\mathcal T_{\Omega_s,\kappa}=\mathcal T_{j,\Omega_s,\kappa}$ and applying \eqref{llu1} we deduce that
\begin{equation}\label{tt1a} K_{1,\Omega_s,1}^{T}=K_{2, \Omega_s,\kappa_0}^{T}.\end{equation}
Let $\mathcal K \subset \Sigma_s$ be compact,
and consider the sets defined in Proposition \ref{prop_convergence_to_pt} with $M$ replaced by $M_s$, $\Gamma$ replaced by $\Sigma_s$. We fix $U(\K)$ the neighborhood of $\K$ in $M_s$ satisfying the properties of  Proposition \ref{prop_convergence_to_pt}. For all $p=(s,y)\in U(\K)$, we define $\B_p(\K) = \B(\Omega_s, \tilde h, T)$ (see the beginning of Section 2.2 for the definition of this set),
where $\tilde h$ is as in Lemma \ref{lem_Ms}
with the choice $h = (s + \epsilon)1_{\Gamma_p(\K)}$, and  $\epsilon$ is as in Proposition \ref{prop_convergence_to_pt}.
For all $H\in C^\infty_0(\B_p(\K))$ and $j=1,2$, we denote by $w_{j,H}$ the solution of \eqref{eq_wave_adjoint2} with $A=A_j$, $q=q_j$.
 We will start by proving that, for all $p=(s,y)\in U(\K) $ there exists $\kappa\in  C^\infty(C_p(\mathcal K))$ such that the following identity holds
\begin{equation}\label{tt1}  w_{1, H}(T,x)=\kappa(x) w_{2,H}(T,x),\end{equation}
for $H\in C^\infty_0(\B_p(\K))$  and $x\in C_p(\mathcal K)$.
For any $x \in C_p(\K)$, we consider functions $h_{k,x}$ and sets $X_{k,x}$ satisfying the properties described in Proposition \ref{prop_convergence_to_pt}   with $\mathcal X = M_s(\Omega_s, \tilde h)^\inter$.
 Analogously to the proof of Theorem \ref{th_main}, we use Lemma \ref{Lem. Tataru} together with Lemma \ref{lem_Ms}, \eqref{blago_inner_prod2} and Proposition \ref{prop_localization}, with  $X_k=X_{k,x}$, to define  $f_{k,x}\in \mathcal C^\infty_0((0,+\infty)\times S$, $k\in\mathbb N$, 
such that
$$\begin{aligned}w_{1, H}(T,x)&=\lim_{k\to+\infty} \left\langle u_{1,f_{k,x}}(T,\cdot), w_{1, H}(T,\cdot)\right\rangle_{L^2(M)}\\
\ &=\lim_{k\to+\infty}\left\langle u_{2,f_{k,x}}(T,\cdot), w_{2,H}(T,\cdot) \right\rangle_{L^2(M)}=\kappa_xw_{2,H}(T,x).\end{aligned}$$
We introduce the function $\kappa: C_p(\K) \to \C$ 
by $\kappa(x) = \kappa_x$, and we get \eqref{tt1}. Moreover,   applying  Lemma \ref{lem_enforce_smoothness}, we deduce that $\kappa$ is smooth and nowhere vanishing in  $C_p(\mathcal K)\cup(\overline{C_p(\mathcal K)}\cap \Sigma_s)$.

We consider now  a collection $\{\K_i;\ i\in I\}$ of compact sets in $\Sigma_s$ and $p_i=(s_i,y_i)$ lying in a neighborhood of $\K_i$ as in Proposition \ref{prop_convergence_to_pt}. We assume that $\{\K_i;\ i\in I\}$ and  $\{p_i;\ i\in I\}$ are chosen in such a way that $\bigcup_{i\in I} C_{p_i}(\K_i) $ is a neighborhood of $\Sigma_s$. Repeating the above argumentation, for all $i\in I$,  we find $\kappa_i\in C^\infty(C_{p_i}(\K_i)\cup(\overline{C_{p_i}(\mathcal K_i)}\cap \Rec))$ such that, for all $x\in C_{p_i}(\K_i)$, we have
\begin{equation}\label{tt2c} w_{1,H}(T,x)=\kappa_i(x) w_{2,H}(T,x),\quad H\in C^\infty_0(\B_{p_i}(\K_i)).\end{equation}
In a similar way to the end of the proof of Theorem \ref{th_main}, applying Lemma \ref{lem_non_vanishing},  we can define $\kappa\in C^\infty\left(\bigcup_{i\in I} C_{p_i}(\K_i)\right)$ such that, for all $i\in I$, $\kappa_{|C_{p_i}(\K_i)}=\overline{\kappa_i}$. Combining this with \eqref{tt2c}, we can define $\mathcal U_s$ a neighborhood of $\Sigma_s$, such that \eqref{llu2} is fulfilled.
\end{proof}

\subsection{Gluing of the gauges}

Let $\Src, \Rec \subset \p M$ satisfy the assumptions of Theorem \ref{th_main_foliation},
and let $\Sigma_s$, $s \in (0,1]$, be a convex foliation satisfying (F1)-(F7). From now on, we assume that \eqref{t1a} is fulfilled and our goal is to prove \eqref{t2a}.
For this purpose, we define the set
\begin{eqnarray}
\label{connected_set_J}
J = \{s \in (0,1];&\text{there exists $U_s \subset M$ an open set of $M$} 
\\\notag&\text{containing $\bar \Omega_s$ such that}
\\\notag&\text{$(A_1|_{U_s},q_1|_{U_s})\in\mathcal G_{U_s, \Rec}(A_2, q_2)$}\}
\end{eqnarray}
According to Theorem \ref{th_main} and condition (F6), we know that $J\neq\emptyset$ since for  $s$ small enough   we have that 
$\overline{\Omega_s} \subset \mathcal U$, where $\mathcal U$ is a neighborhood of $\Rec$ as in Theorem \ref{th_main}.
Moreover, the continuity condition (F5) implies that $J$ is open.  Therefore, since $(0,1]$ is a connected set, the proof of 
Theorem \ref{th_main_foliation} will be completed if we show that $J$ is closed. This will be our main task from now on. We start with four intermediate results.

%In order to show this, it is enough to show that if $s \in J$, then $\Lambda_{\Src, \Rem}^{\infty}$ determines $g$ and $(A, Q)$ in a neighborhood $U_s$ of $\Sigma_s$.
%Indeed, the openness follows then from the continuity (4) since  $\dist(\Omega_s, M \setminus (U_s \cup \Omega_s \cup \mathcal U))$ is strictly positive. 

%\begin{lemma}
%\label{lem_kappa_smooth}
%Let $U \subset M^\inter$ be open and suppose that $\bar U \cap \p M \subset \Rec$.
%Suppose that $\kappa : \bar U \to \C$ is smooth near $\Rec$
%and that $\mathcal T_{U, \kappa} f \in C^\infty(\bar U)$
%for all $f \in C_0^\infty((0,\infty) \times \Src)$.
%Then $\kappa \in C^\infty(\bar U)$.
%\end{lemma}

Let $U \subset M^\inter$ be open. We define
\def\KK{\mathbb K}
$$
\KK(U) := \{\kappa \in C^\infty(\bar U);\ \kappa|_{\bar U \cap \Rec} = 1,\ \kappa(x) \ne 0,\ x \in \bar U \}.
$$

\begin{lemma}
\label{lem_moving_data}
Let $U \subset M^\inter$ be open and connected and suppose that $\bar U \cap \p M \subset \Rec$ and that the interior of $\bar U \cap \Rec$ 
in $\p M$ is nonempty. Assume that there exists a piecewise smooth function $\kappa_0 : U \to \C$ with the following properties:
\begin{itemize}
\item[(i)] $\kappa_0(x) \ne 0$ for all $x \in U$,
\item[(ii)] there is a neighbourhood $W \subset M$ of $\Rec$
such that $\kappa_0$ is smooth in $U \cap W$
and extends smoothly to $\overline U \cap \Rec$,
\item[(iii)] the smooth extension satisfies $\kappa_0 = 1$ in $\bar U \cap \Rec$.
\end{itemize}
Suppose that, for $\mathcal A_j=\Delta_g-A_j-q_j$, $j=1,2$,  the conditions (\ref{t1a}) and 
\begin{equation}\label{lmda}\mathcal A_1|_U=\kappa_0^{-1}\mathcal A_2\kappa_0|_U\end{equation}
are fulfilled. Then $\kappa_0$ is smooth and has a smooth extension $\kappa$ to $\overline U$.
The smooth extension satisfies $\kappa \in \KK(U) $ and the condition
\begin{equation}\label{lmdb_mod}
\mathcal T_{1,U,1}=\mathcal T_{2,U,\kappa}\end{equation}
is fulfilled.
\end{lemma}
\begin{proof}
We will divide the proof in four steps. 

Step 1. We will show that (\ref{lmdb_mod})
holds with $U$ replaced by a small subset of $U$ lying close to $\Rec$. 
As the interior of $\bar U \cap \Rec$ 
in $\p M$ is nonempty, we can choose a nonempty open set $\Gamma \subset \bar U \cap \Rec$ and $r > 0$ such that 
$M(\Gamma, r) \subset U$. For $j=1,2$, $f\in C_0^\infty((0,+\infty) \times S)$ and $\phi \in C_0^\infty((T-r,T) \times \Gamma)$  we consider $u_{j,f}$, $v_{j,\phi}$ solving respectively  \eqref{eq_wave}, \eqref{eq_wave_adjoint1} with $A=A_j$ and $q=q_j$. By the finite speed of propagation (e.g. \cite[Lemma 3.9]{Katchalov2001}), we know that, for all $t\in[0,T]$,  supp$(v_{2,\phi}(t,\cdot))\subset M(\Gamma, r)$. Then, \eqref{lmda} and the uniqueness of solutions of \eqref{eq_wave_adjoint1} imply \begin{equation}\label{lmdc}v_{1,\phi}=(\overline{\kappa_0})^{-1}v_{2,\phi}.\end{equation} 
 By Lemma \ref{lem_blago}, the condition \eqref{t1a} implies that
$$
\bra v_{1,\phi}(T,\cdot),u_{1,f}(T,\cdot)\cet_{L^2(M)} 
= \bra v_{2,\phi}(T,\cdot),u_{2,f}(T,\cdot)\cet_{L^2(M)} 
$$
and combining this with \eqref{lmdc} we get
$$
\bra v_{1,\phi}(T,\cdot),u_{1,f}(T,\cdot)-\kappa_0u_{2,f}(T,\cdot)\cet_{L^2(M)} =0.
$$
Then, using the density of the functions $v_{1,\phi}(T,\cdot)$, $\phi \in C_0^\infty((T-r,T) \times \Gamma)$, in $L^2(M(\Gamma, r))$, given by Lemma \ref{Lem. Tataru},  we find for $x \in M(\Gamma, r)$ that 
    \begin{equation}\label{u1kappau2}
u_{1,f}(T,x)=\kappa_0u_{2,f}(T,x).    
    \end{equation}
Now allowing $T>0$ to be arbitrary, we deduce that
\begin{equation}\label{lmdd}\mathcal T_{1,B,1}=\mathcal T_{2,B,\kappa_0},\end{equation}
where $B = M(\Gamma, r)^\inter$.

Step 2. 
Supposing that $p \in U$ and $\epsilon > 0$ satisfy $B(p, 2\epsilon) \subset U$,
we will show that 
\begin{equation}\label{lmde}\mathcal T_{1,B(p,\epsilon),1}=\mathcal T_{2,B(p,\epsilon),\kappa_0},\end{equation}
implies
\begin{equation}\label{lmdf}\mathcal T_{1,B(p,2\epsilon),1}=\mathcal T_{2,B(p,2\epsilon),\kappa_0}.\end{equation}

For $H\in C^\infty_0((0,+\infty)\times  B(p,\epsilon))$, we consider  $w_{j,H}$ solving  \eqref{eq_wave_adjoint2} with $A=A_j$ and $q=q_j$. Applying Lemma \ref{lem_blago} and \eqref{lmde}, for $ H \in C_0^\infty((0,\infty) \times B(p,\epsilon)),\ 
f \in C_0^\infty((0,\infty) \times \Src)$, we get
$$
\bra w_{1, H}(T,\cdot),u_{1,f}(T,\cdot)\cet_{L^2(M)} 
= \bra w_{2,\overline{\kappa_0} H}(T,\cdot),u_{2,f}(T,\cdot)\cet_{L^2(M)} .
$$
On the other hand, for $H \in C_0^\infty((T-\epsilon, T) \times B(p,\epsilon))$, the finite speed of propagation and \eqref{lmda} imply $w_{2,\overline{\kappa_0} H}(T,\cdot)=\overline{\kappa_0} w_{1, H}(T,\cdot)$.
Then we have
$$
\bra w_{1, H}(T,\cdot),u_{1,f}(T,\cdot)-\kappa_0 u_{2,f}(T,\cdot)\cet_{L^2(M)} =0
$$
and using the density of  the functions $w_{1, H}(T,\cdot)$, $H \in C_0^\infty((T-\epsilon, T) \times B(p,\epsilon))$, in $L^2(M(B(p,\epsilon), \epsilon))$, given by Lemma \ref{Lem. Tataru},  we find
$$u_{1,f}(T,x)=\kappa_0u_{2,f}(T,x),\quad x\in B(p,2\epsilon).$$ 
Allowing $T>0$ to be arbitrary, we get \eqref{lmdf}. 

Step 3. We will show that $\mathcal T_{1,U,1}=\mathcal T_{2,U,\kappa_0}$. 
Let $p \in U$ and $p' \in M(\Gamma,r)^\inter$
and connect $p$ to $p'$
with a path $\gamma : [0,1] \to U$.
Then there is $\epsilon > 0$ such that 
$B(\gamma(t), 2 \epsilon) \subset U$ for all $t \in [0,1]$
and $B(\gamma(0), \epsilon) \subset M(\Gamma,r)^\inter$.
Now we can iteratively prove that
$$\mathcal T_{1,B(\gamma(t),\epsilon),1}=\mathcal T_{2,B(\gamma(t),\epsilon),\kappa_0},\quad t\in[0,1].$$
Since $p \in U$ can be chosen arbitrarily, we deduce that $\mathcal T_{1,U,1}=\mathcal T_{2,U,\kappa_0}$.

Step 4. We will show that $\kappa_0$ is smooth and that it has a smooth, nowhere vanishing extension to $\overline U$.
Let $x_0 \in \bar U \cap M^\inter$.
By Lemma \ref{lem_non_vanishing} there is $f \in C_0^\infty((0,\infty) \times \Src)$ and a neighborhood $B \subset M^\inter$ of $x_0$ 
such that $u_{2,f}(T,x) \ne 0$ for $x \in B$.
Therefore, 
$$
\kappa_0(x) = \frac{u_{1,f}(T,x)}{u_{2,f}(T,x)}, \quad x \in B \cap U.
$$
As $u_{j,f}(T,\cdot) \in C^\infty(\overline U)$
for both $j=1,2$,
this implies that $\kappa_0$ is smooth in $B \cap U$ and has a smooth extension to $B \cap \bar U$.
By varying $x_0$, we see that $\kappa_0$ has a smooth extension to $\overline U \cap M^\inter$. Recalling also the assumption (ii), we see that $\kappa_0$ has a smooth extension to whole $\overline U$.
The extension is unique and we denote it by $\kappa$.
The smoothness of $\kappa$ and (\ref{u1kappau2}) 
for $x \in U$
imply that 
    \begin{equation}\label{u1kappau2ext}
u_{1,f}(T,x)=\kappa(x) u_{2,f}(T,x),
    \end{equation}
for all $x \in \overline U$ and all $f \in C_0^\infty((0,\infty) \times \Src)$.
To see that $\kappa$ is nowhere vanishing in $\overline U$,
let $x \in \bar U \cap M^\inter$
and choose $f \in C_0^\infty((0,\infty) \times \Src)$ 
such that $u_{1,f}(T,x) \ne 0$ (using Lemma \ref{lem_non_vanishing} again). Now (\ref{u1kappau2ext}) implies that $\kappa(x) \ne 0$.
\end{proof}
A direct consequence of this last result is given by the following.

\begin{corollary}
\label{cor_moving_data_R}
Let $s \in J$ where $J$ is defined by (\ref{connected_set_J}). Then
\eqref{t1a} implies that there exists $\kappa_s \in \KK(\Omega_s)$ such that
$$\mathcal A_1|_{\Omega_{s}}=\kappa_s^{-1}\mathcal A_2\kappa_s|_{\Omega_{s}},\quad \mathcal T_{1,\Omega_{s},1}=\mathcal T_{2,\Omega_{s},\kappa_s}.$$
\end{corollary}
Let us also consider the following result which will be important for the gluing of the gauge class.

\begin{lemma}
\label{lem_moving_data_glueing}
Let $s_1, s_2 \in J$  where $J$ is defined by (\ref{connected_set_J}), and suppose that $s_1 < s_2$.
Then \eqref{t1a} implies that there exist $\kappa_\ell \in \KK(\Omega_{s_\ell})$, $\ell =1,2$, such that $\kappa_2|_{\Omega_{s_1}} = \kappa_1$ and such that the condition 
\begin{equation}\label{lmdga}\mathcal A_1|_{\Omega_{s_\ell}}=\kappa_\ell^{-1}\mathcal A_2\kappa_\ell|_{\Omega_{s_\ell}},\quad \mathcal T_{1,\Omega_{s_\ell},1}=\mathcal T_{2,\Omega_{s_\ell},\kappa_\ell},\ \ell=1,2\end{equation}
is fulfilled.
\end{lemma}
\begin{proof}
By Corollary \ref{cor_moving_data_R} there exist $\kappa_\ell \in \KK(\Omega_{s_\ell})$, $\ell =1,2$, such that \eqref{lmdga} is fulfilled and the proof will be completed if we can show that $\kappa_2|_{\Omega_{s_1}} = \kappa_1$. For this purpose, we remark that \eqref{lmdga} implies
$$
\mathcal T_{2,\Omega_{s_1},\kappa_1} f(x) = \mathcal T_{2,\Omega_{s_2},\kappa_2} f(x),
\quad f \in C_0^\infty((0,\infty) \times \Src),\ x \in \Omega_{s_1}.
$$ 
We fix $x \in \Omega_{s_1}$.
By Lemma \ref{lem_non_vanishing} there is $f \in C_0^\infty((0,\infty) \times \Src)$ and a neighborhood $B \subset M^\inter$ of $x$ 
such that $u_{2,f}(T,x) \ne 0$ in $B$, where $u_2$ is the solution of (\ref{eq_wave}) for $A=A_2$ and $q=q_2$.
Thus $\kappa_1 = \kappa_2$ in $B\cap \Omega_{s_1}$ and allowing $x \in \Omega_{s_1}$ to be arbitrary, we deduce that $\kappa_2|_{\Omega_{s_1}} = \kappa_1$. This completes the proof of the lemma.
\end{proof}

\begin{lemma}
\label{lem_moving_data_induction}
Let $(s_\ell)_{\ell\in\mathbb N}$ be a strictly increasing sequence  of $\R$ such that $s_\ell \in J$, $\ell\in\mathbb N$, and 
suppose that $\lim_{\ell \to \infty} s_\ell = s$. Here $J$ is defined by (\ref{connected_set_J}).
Then \eqref{t1a} implies that there exists $\kappa\in\KK(\Omega_{s})$ such that
$$\mathcal A_1|_{\Omega_{s}}=\kappa^{-1}\mathcal A_2\kappa|_{\Omega_{s}},\quad \mathcal T_{1,\Omega_{s},1}=\mathcal T_{2,\Omega_{s},\kappa}.$$
\end{lemma}
\begin{proof}
An induction using Corollary \ref{cor_moving_data_R} and Lemma \ref{lem_moving_data_glueing}
shows that, for all $\ell\in\mathbb N$,  there exists $\kappa_\ell \in \KK(\Omega_{s_\ell})$,  such that  the following conditions are fufilled 
\begin{equation}\label{lmdia}\mathcal A_1|_{\Omega_{s_\ell}}=\kappa_\ell^{-1}\mathcal A_2\kappa_\ell|_{\Omega_{s_\ell}},\quad \mathcal T_{1,\Omega_{s_\ell},1}=\mathcal T_{2,\Omega_{s_\ell},\kappa_\ell},\quad \ell\in\mathbb N,\end{equation}
\begin{equation}\label{lmdib}\kappa_{\ell+1}|_{\Omega_{s_\ell}} = \kappa_\ell,\quad \ell\in\mathbb N.\end{equation}

According to \eqref{lmdib}, the functions $\kappa_\ell$, $\ell\in\mathbb N$, fit together and give a function $\kappa_\infty$ on $\Omega_s$ defined by
$$\kappa_\infty|_{\Omega_{s_\ell}} = \kappa_\ell,\quad \ell\in\mathbb N.$$
Moreover, by \eqref{lmdia} we have
$$\mathcal A_1|_{\Omega_{s}}=\kappa_\infty^{-1}\mathcal A_2\kappa_\infty|_{\Omega_{s}}.$$
Here we recall that $\kappa_\infty$ is  smooth in $\Omega_s$, up to $\overline{\Omega_s} \cap \Rec$, $\kappa_\infty$ is nowhere vanishing in $\Omega_s$ and satisfies $\kappa_\infty = 1$ in $\overline{\Omega_s} \cap \Rec$.
Now Lemma \ref{lem_moving_data} 
implies that $\kappa_\infty$ has a smooth extension $\kappa \in \KK(\Omega_s)$.
\end{proof}

Armed with the above lemmas, we are now in position to complete the proof of the global result.

\begin{proof}[Proof of Theorem \ref{th_main_foliation}]
It remains to show that $J$ is closed. 
Let $(s_\ell)_{\ell\in\mathbb N}$ be a strictly increasing sequence  of $\R$ such that $s_\ell \in J$, $\ell\in\mathbb N$, and 
suppose that $\lim_{\ell \to \infty} s_\ell = s$. We will show that $s \in J$.
By Lemma \ref{lem_moving_data_induction}, there exists $\kappa_0 \in \KK(\Omega_s)$ such that
\begin{equation}\label{lmdia2}\mathcal A_1|_{\Omega_{s}}=\kappa_0^{-1}\mathcal A_2\kappa_0|_{\Omega_{s}},\quad \mathcal T_{1,\Omega_{s},1}=\mathcal T_{2,\Omega_{s},\kappa_0}.\end{equation}
Combining this with Lemma \ref{lem_lot_Us}, we deduce that there exist a neighborhood $\mathcal U_s\subset M_s$ of $\Sigma_s$ and $\kappa_1\in C^\infty(\mathcal U_s)$ such that
$$%\label{lmdia}
\mathcal A_1|_{\mathcal U_{s}}=\kappa_1^{-1}\mathcal A_2\kappa_1|_{\mathcal U_{s}}.$$
Combining this with \eqref{lmdia2}, we obtain that for $\kappa$ defined by
$$
\kappa(x) = \begin{cases}
\kappa_0(x), & x \in \Omega_s,
\\
\kappa_1(x), & x \in \mathcal U_s,
\end{cases}
$$
and for $U_s=\Omega_s\cup\mathcal U_{s}$
we have 
$$%\label{lmdia}
\mathcal A_1|_{U_{s}}=\kappa^{-1}\mathcal A_2\kappa|_{U_{s}}.$$
It is immediate that $\kappa$
is piecewise smooth,
and Lemma \ref{lem_moving_data} implies then 
that $\kappa$ is in $\KK(U_s)$. Thus $J$ is closed and since $J$ is also open in $(0,1]$, we deduce that $J=(0,1]$. In particular $1\in J$ which completes the proof.
\end{proof}

\section{Complementary results}
\label{sec_complementary_results}

In this section we show that instead of assuming exact controllability from $\Src$ and strict convexity of $\Rec$,
we may assume that exact controllability holds from either $\Src$ or $\Rec$ 
and that one of them is strictly convex. 
In the case that exact controllability holds from the set that is also strictly convex, we need the additional assumption that all the points in $M$ can be reached from the other set in time $T$.
More precisely, supposing that 
exact controllability holds from strictly convex $\Rec$, we assume that
    \begin{equation}\label{dist_S}
T > \max_{x \in M} d(x,\Src).
    \end{equation} 
Then we can determine $A$ and $q$, up to the gauge equivalence, near the strictly convex set $\Rec$ or $\Src$.

Observe first that the adjoint of $\Lambda^T_{\Src,\Rec}$
is $R \Lambda^T_{\Rec,\Src}R$ where $R$ is the time-reversal 
$R\phi(t) = \phi(T-t)$.
Thus Theorem \ref{th_main} implies that we can determine the  geometry and the lower order terms near $\Src$ if it is strictly convex and exact controllability holds from $\Rec$.

We will show next that the conclusion of Theorem \ref{th_main} holds when $\Rec$ is strictly convex, the wave equation (\ref{eq_wave}) is exactly controllable from $\Rec$, and (\ref{dist_S}) holds.
The fourth case, that is, exact controllability holds from  strictly convex $\Src$ and (\ref{dist_S}) holds with $\Src$ replaced by $\Rec$, follows then again 
by transposition.
The global uniqueness result in Theorem \ref{th_main_foliation} can also be changed in the analogous manner.
 
We used the exact controllability only once in the proof of Theorem \ref{th_main},
namely 
when we invoked Proposition
\ref{prop_localization}. 
Proposition \ref{prop_localization_Rec} below
will substitute Proposition \ref{prop_localization} in the case when 
the exact controllability holds from $\Rec$ instead of from $\Src$.

\begin{proposition}
\label{prop_localization_Rec}
Let $\mathcal X \subset M$ be open, $x \in \mathcal X$ and let $X_k \subset M$, $k\in\mathbb N$, be a sequence of neighborhoods of $x$ satisfying 
$\lim_{k \to \infty}\diam(X_k) = 0$. % as $j \to \infty$.
Let $\psi_0 \in C_0^\infty(\mathcal X)$ satisfy $\psi_0(x) \ne 0$.
Let $T >0$ and suppose that a sequence $(f_{kl})_{k,l=1}^\infty$ of functions in $L^2((0,T) \times \Src)$ satisfies
\begin{itemize}
\item[(0)] for all $k$, there is $\delta_k \in L^2(M)$ such that the sequence $(u_{f_{kl}}(T,\cdot))_{l = 1}^\infty$
converges weakly to $\delta_k$ in $L^2(M)$,
\item[(i)] there is $C > 0$ such that $\norm{\delta_k}_{L^2(M)} \le C |X_k|^{-1/2}$ for all $k$,
\item[(ii)] $\supp(\delta_k) \subset \bar X_k \cup (M \setminus \mathcal X)$
for all $k$,
\item[(iii)] $(\bra \delta_k, \psi_0 \cet_{L^2(M)} )_{k=1}^\infty$ converges.
\end{itemize}
Then there is $\kappa \in \C$ such that $\underset{k \to \infty}{\lim} \underset{l \to \infty}{\lim}\bra u_{f_{kl}}(T,\cdot), \psi \cet_{L^2(M)} =\kappa \psi(x)$ for all $\psi \in C_0^\infty(\mathcal X)$.

Furthermore, if (\ref{dist_S}) holds, then there is a sequence $(f_{kl})_{k,l=1}^\infty$ that satisfies (i)-(iii) and for which $\kappa = 1$.
\end{proposition}

\begin{proof}
Let $\psi \in C_0^\infty(\mathcal X)$.
Then $\supp(\delta_k \psi) \subset \bar X_k$ and
%We may perform the below computations in coordinates around $q$.
%We have
$$
\bra \delta_k, \psi\cet_{L^2(M)} = \psi(x) \bra \delta_k, 1\cet_{L^2(X_k)} + R_k,
$$
where the remainder term $R_k$ converges to zero as $k \to \infty$.
This can be seen as in the proof of Proposition \ref{prop_localization}
since $\norm{\delta_k} \le C |X_k|^{-1/2}$ for all $k$.
We choose $\psi = \psi_0$ and see that $\lim_{k \to \infty} \bra \delta_k, 1\cet_{L^2(X_k)}$
exists. We denote the limit by $\kappa$. Thus for any $\psi \in C_0^\infty(\mathcal X)$ it holds that 
$$
\underset{k \to \infty}{\lim} \underset{l \to \infty}{\lim} \bra K_{\Rec}^T f_{kl}, \psi \cet_{L^2((0,T) \times \Rec)}
= \lim_{k\to \infty}
\bra \delta_k, \psi\cet_{L^2(M)} = \kappa \psi(x).
$$

Let us now assume that 
$T > \underset{x \in M}{\max} d(x,\Src)$. Then Lemma \ref{Lem. Tataru} 
 implies that for each $k$ there is a sequence 
$(f_{kl})_{l=1}^\infty$
in $L^2((0,T) \times \Src)$ such that 
$(u_{ f_{kl}}(T,\cdot))_{l=1}^\infty$ converges to $1_{X_k} / |X_k|$ in $L^2(M)$.
Then the conditions (0), (ii) and (iii) hold.
\end{proof}

Let us now outline how the proof of Theorem \ref{th_main}
needs to be changed when Proposition
\ref{prop_localization} is replaced by Proposition \ref{prop_localization_Rec}.
Let $x \in C_p(\K)$ let $X_{k,x}$ and $\mathcal X$
be as in the proof of Theorem \ref{th_main}.
We use the shorthand notation $X_k = X_{k,x}$
and do not emphasize the dependence on $x$ in the notation below.

As in the proof of Proposition \ref{prop_localization_Rec},
let $(f_{kl})_{k,l=1}^\infty$
in $L^2((0,T) \times \Src)$ be such that 
$(u_{1,f_{kl}}(T,\cdot))_{l=1}^\infty$ converges to $1_{X_k} / |X_k|$ in $L^2(M)$.
Note that
%, for $j=1$, the conditions (i)-(iii) of Lemma \ref{lem_localization} are fulfilled and we have 
\begin{equation*}%\label{th2mod} 
\lim_{k\to+\infty} \lim_{l\to+\infty} \left\langle u_{1,f_{kl}}(T,\cdot), \psi \right\rangle_{L^2(M)}=\psi(x),\quad \psi\in C^\infty_0(\mathcal X ).\end{equation*}
We will show that the conditions (0)-(iii) of Proposition \ref{prop_localization_Rec} are fulfilled with respect to $j=2$.
Then Proposition \ref{prop_localization_Rec} implies the following analogue of equation (\ref{th2b}) in the proof of Theorem \ref{th_main},
    \begin{equation*}
\underset{k \to \infty}{\lim} \underset{l \to \infty}{\lim}\bra u_{2,f_{kl}}(T,\cdot), \psi \cet_{L^2(M)} =\kappa \psi(x), \quad \psi\in C^\infty_0(\mathcal X ).
    \end{equation*}
The rest of the proof of Theorem \ref{th_main} is unchanged.

Analogously to (\ref{th2a}), we have
\begin{equation}\label{th2a_mod} \begin{aligned}\left\langle u_{1,f_{kl}}(T,\cdot), v_{1,\phi}(T,\cdot) \right\rangle_{L^2(M)}
=\left\langle u_{2,f_{kl}}(T,\cdot), v_{2,\phi}(T,\cdot) \right\rangle_{L^2(M)},\end{aligned}\end{equation}
for $\phi \in L^2((0,T) \times \Rec)$.
Therefore
    \begin{equation}\label{weak_conv}
\lim_{l \to \infty} 
\left\langle u_{2,f_{kl}}(T,\cdot), v_{2,\phi}(T,\cdot) \right\rangle_{L^2(M)} = 
\left\langle1_{X_k} / |X_k|, v_{1,\phi}(T,\cdot) \right\rangle_{L^2(M)}.
    \end{equation}
Recall that Hilbert spaces are sequentially weakly complete,
see e.g. \cite[Th. V.7]{Yosida}.
Assuming exact controllability from $\Rec$, we have
    \begin{equation}\label{exact_L2}
\{v_{2,\phi}(T,\cdot);\ \phi \in L^2((0,T) \times \Rec)\}
= L^2(M),
    \end{equation}
and hence (\ref{weak_conv}) implies that $(u_{2,f_{kl}}(T,\cdot))_{l \in \N}$ converges weakly in $L^2(M)$.
This establishes (0), and we denote the weak limit by $\delta_k$.

It follows from (\ref{weak_conv}), using the Cauchy-Schwarz inequality, that 
\begin{equation*}\begin{aligned}
\left\langle |X_k|^{1/2} \delta_k, v_{2,\phi}(T,\cdot) \right\rangle_{L^2(M)} 
&= 
\left\langle1_{X_k} / |X_k|^{1/2}, v_{1,\phi}(T,\cdot) \right\rangle_{L^2(M)}
\\&\le \norm{v_{1,\phi}(T,\cdot)}_{L^2(M)}.
\end{aligned}\end{equation*}
Thus (\ref{exact_L2}) together with the uniform boundedness theorem implies that 
$(|X_k|^{1/2} \delta_k)_{k \in \N}$ is bounded in $L^2(M)$.
Therefore (i) holds. 
Using again (\ref{weak_conv}), we have
    \begin{equation*}
\left\langle \delta_k, v_{2,\phi}(T,\cdot) \right\rangle_{L^2(M)} = 
\left\langle1_{X_k} / |X_k|, v_{1,\phi}(T,\cdot) \right\rangle_{L^2(M)} = 0
    \end{equation*}
for all $\phi \in C_0^\infty(\B(\Rec, h_{k,x}; T))$, 
since $v_{1,\phi}(T,\cdot)$ is supported in 
the set $M(\Rec, h_{k,x})$ that is disjoint with $X_k$,
see Proposition \ref{prop_convergence_to_pt}.
Thus the density result of Lemma \ref{Lem. Tataru} implies (ii). 
Using once more (\ref{weak_conv}), and recalling that $\mathcal X = M(B_{\partial M}(y,\epsilon), s+\epsilon)^\inter$, we have
for $\phi\in C^\infty_0(\B(B_{\partial M}(y,\epsilon), s+\epsilon;T ))$,
$$
\lim_{k \to \infty}
\lim_{l \to \infty} 
\left\langle u_{2,f_{kl}}(T,\cdot), v_{2,\phi}(T,\cdot) \right\rangle_{L^2(M)} = v_{1,\phi}(T,x).
$$
Thus, condition (iii) of Proposition \ref{prop_localization_Rec} is also fulfilled when $\phi$ is chosen so that $v_{2,\phi}(T,x) \ne 0$.

\medskip

\noindent{\bf Acknowledgements.} 
Y.\ Kian was supported by  the French National Research Agency ANR (project MultiOnde) grant ANR-17-CE40-0029,
M.\ Lassas by Academy of Finland, projects 263235, 273979 and 284715,
Y.\ Kurylev by the Engineering and Physical Sciences Research Council (EPSRC), UK, grant EP/L01937X/1,
and L.\ Oksanen by the EPSRC grant EP/L026473/1.

\end{document}